\documentclass{amsart}
\usepackage{amsmath}
\usepackage{amsfonts}
\usepackage{amssymb}
\usepackage{epsfig}
\newtheorem{theorem}{Theorem}[section]
\newtheorem{theo}{Theorem}
\newtheorem{lm}[theorem]{Lemma}
\newtheorem{tr}[theorem]{Theorem}
\newtheorem{cor}[theorem]{Corollary}

\newtheorem{rem}[theorem]{Remark}
\newtheorem{pr}[theorem]{Proposition}
\usepackage{color}

\begin{document}
\title{Blocks for general linear supergroup $GL(m|n)$}
\thanks{This publication was made possible by a NPRF award NPRP 6 - 1059 - 1 - 208 from the Qatar National Research Fund 
(a member of The Qatar Foundation). The statements made herein are solely the responsibility of the authors.}
\author{Franti\v sek Marko}
\address{Pennsylvania State University \\ 76 University Drive, \\ Hazleton, PA 18202, \\ USA}
\email{fxm13@psu.edu}
\author{Alexandr N. Zubkov}
\address{Omsk State Polytechnic University, Mira 11, 644050, Russia}
\email{a.zubkov@yahoo.com}
\begin{abstract}
We prove the linkage principle and describe blocks of the general linear supergroups $GL(m|n)$ over the ground field $K$ of characteristic $p\neq 2$.
\end{abstract}
\maketitle

\section*{Introduction}

One of the important problems in representation theory is to describe simple objects with nontrivial extensions.
For algebraic groups, a necessary condition for the existence of nontrivial extensions of simple modules is given by the linkage principle.
It was conjectured by Verma in \cite{verma} that the highest weights $\mu$ of all simple composition factors of the induced module with highest weight $\lambda$ belong 
to the orbit of $\lambda$ under the dot action of the affine Weyl group. In special cases this conjecture was proved by Jantzen in \cite{jan1} and \cite{jan2}. 
The Verma conjecture and the linkage principle in general were proved by Andersen in \cite{and}.  
Linkage principle is essential for the description of blocks. The complete description of blocks for semisimple algebraic groups was obtained by Donkin in \cite{don}. More details about the linkage principle and blocks of algebraic groups can be found in section II.6. of the book \cite{jan}. 

The goal of our paper is to generalize the linkage principle and to describe blocks of general linear supergroups $G=GL(m|n)$ over perfect fields $K$ of positive characteristic $p\neq 2$. 

The notions of even and odd linkages of weights of $G$ were first introduced in \cite{marko}. Two dominant weights $\lambda$ and $\mu$ of $G$ will be called {\it even-linked} if they belong to the same block of the category of rational $G_{res}=GL(m)\times GL(n)$-modules. Further, $\lambda$ and $\mu$ will be called {\it simply-odd-linked} if there is an odd positive root 
$\alpha$ such that $p|(\lambda+\rho, \alpha)$ and $\mu=\lambda\pm\alpha$, where $\rho$ and the bilinear form $(,)$ are defined in Section \ref{3}.
The following main result of this article was stated as Conjecture 5.10 in \cite{marko}. 
\begin{theo}\label{main}
Dominant weights $\lambda$ and $\mu$ of $G$ belong to the same block of $G$ if and only if there is a chain
of dominant weights $\lambda=\lambda_0, \lambda_1, \ldots, \lambda_t=\mu$ such that for each $i$ the weights $\lambda_i, \lambda_{i+1}$ are either even-linked or simply-odd-linked. 
\end{theo}
Blocks of $G$, in the case when the characteristic of the ground field is zero, are described as follows. First of all, the category of rational $G$-supermodules can be identified with the category of {\it integrable} supermodules over its distribution superalgebra $Dist(G)$ (cf. \cite{brunkuj}, Corollary 3.5). Since, over a field of characteristic zero, $Dist(G)$ is isomorphic to the universal enveloping superalgebra of Lie superalgebra $Lie(G)=gl(m|n)$ of $G$ (cf. \cite{zub3}, Lemma 3.1), one can work in the category of finite-dimensional $Lie(G)$-supermodules.

According to an unpublished work of Vera Serganova (see \cite{verserg}, Theorem 3.7) all central blocks of classical Lie superalgebras are indecomposable. 
Therefore, all central blocks of classical Lie superalgebras coincide with regular blocks defined using extensions of simple modules. Proposition 3.1 from \cite{verserg} completes a description of blocks (see also \cite{chengwang}, Theorem 2.27). This description can be derived also from the description of multiplicities of irreducible supermodules in Kac's supermodules, which was given in \cite{grusserg}. A special case, when the degree of atypicality of a block over special linear Lie superalgebra is $1$, was considered by Jerome Germoni (see \cite{germ}).

The article is organized as follows. In the first two sections we prove auxiliary results that will be used later to superize some classical results. 
For example, Proposition \ref{alifttoGH'} and Proposition \ref{thesameforGH'} superize Proposition I.8.20 of \cite{jan} under an additional assumption that the finite supergroup $G$ is infinitesimal. If $G$, as a right $H$-superscheme, is split as $G=W\times H$ and $G$ is infinitesimal, then using Remark \ref{asuffcond} we can prove a stronger statement that
the $G$-supermodule structure of $coind^G_H M|_H$ can be lifted to the $GH'$-supermodule structure of $coind^{GH'}_{H'} M=Dist(GH')\otimes_{Dist(H')} M$ for any 
$H'$-supermodule $M$.

In the third section we discuss positive and negative parts of the root system of $GL(m|n)$ with respect to different choices of Borel supersubgroups. 
In the fourth section, we superize results from sections II.3 and II.9 of \cite{jan} to obtain statements about representations of supersubgroups 
$G_r, G_r T$ and $G_r B_w^{\pm}$ of $G$. The most useful results are Proposition \ref{duality} that superizes Lemma II.9.2 of \cite{jan}, and 
Lemmas \ref{characters}, \ref{hom-s},  \ref{atopof} that superize Lemmas 2, 3, 4 of \cite{sdoty}.

The fifth section is devoted to representations of minimal parabolic supersubgroups. The specific structure of induced supermodules over such supergroups was used in \cite{pen} to prove a superanalog of Borel-Bott-Weil theorem over a field of characteristic zero (see also  \cite{dem, penskor, zub1}). 

In the sixth section we use an analogous idea to modify the celebrated proof of the strong linkage principle for $G_r T$, given in \cite{sdoty}, in a way that also involves odd linkage.
In the seventh section we describe all blocks over $GL(m|n)$. The necessary conditions of Theorem \ref{main} can be derived from the strong linkage principle proved in the previous section.
The sufficient condition can be proved for even-linked and for simply-odd-linked weights separately. The first case has been already proved in \cite{marko}. We prove the second case using an interesting trick involving certain $GL(m)\times GL(n)$-module morphism $\phi_1$ on the first floor $F_1$ of the induced supermodule $H^0(\lambda)$. 

\section{Algebraic supergroups and their representations}

In this section we introduce algebraic supergroups, their basic properties and the concept of blocks. Then we formulate suitable results about actions of supergroups and their
distribution algebras on various objects.

An algebraic supergroup $G$ is a representable functor from the category of supercommutative superalgebras $\mathsf{SAlg}_K$ to the category of groups $\mathsf{Gr}$, i.e.
for every $A\in\mathsf{SAlg}_K$ we have $G(A)=Hom_{\mathsf{SAlg}_K}(K[G], A)$, where $K[G]$ is a finitely generated {\it coordinate superalgebra} of $G$ (cf. \cite{zub4}).

Let $G$ be an algebraic supergroup defined over a perfect field $K$ of characteristic $p\neq 2$. A Hopf superalgebra structure on $K[G]$ is defined by the comultiplication $\Delta_{K[G]}$, the antipode $s_{K[G]}$ and the counit $\epsilon_{K[G]}$. A closed supersubgroup $H$ of $G$ is uniquely defined by a Hopf superideal $I_H$ such that an element $g\in G(A)$ belongs to $H(A)$ if and only if $g(I_H)=0$ for $A\in\mathsf{SAlg}_K$. Therefore, $K[H]\simeq K[G]/I_H$.  For example, the largest even supersubgroup $G_{ev}$ is defined by the Hopf superideal $K[G]K[G]_1$, where $K[G]_1$ is the odd component of $K[G]$, and its coordinate superalgebra $K[G_{ev}]$ is $K[G]/K[G]K[G]_1$. Since $K[G]/K[G]K[G]_1$ is an even Hopf superalgebra, it also represents an algebraic group, that is denoted by $G_{res}$. 

A group of characters of $G$, which will be denoted by $X(G)$, consists of all group-like elements of $K[G]$.

Let $G-\mathsf{SMod}$ (and $\mathsf{SMod}-G$, respectively) denote the category of all (not necessarily finite-dimensional) rational left (right, respectively) supermodules over $G$ considered together with the ungraded homomorphisms between them.
The category $G-\mathsf{SMod}$ coincides with the category of right $K[G]$-supercomodules over the Hopf superalgebra $K[G]$, whose morphisms are $K[G]$-comodule morphisms. If $V\in G-\mathsf{SMod}$, then the corresponding (super)comodule map is denoted by $\tau_V : V\to V\otimes K[G]$.  An analogous setup is valid for $\mathsf{SMod}-G$ 
(see \cite{zub4} for more details). Both categories $G-\mathsf{SMod}$ and $\mathsf{SMod}-G$ are not abelian but their underlying {\it even} categories, consisting of the same objects and even morphisms, are. Denote these abelian categories by $G-\mathsf{Smod}$ and $\mathsf{Smod}-G$, respectively.  

For a given $G$-supermodule $V$ let $\Pi V$ denote the $G$-supermodule that coincides with $V$ as a $K[G]$-comodule but its superspace structure is defined by $(\Pi V)_0=V_1, (\Pi V)_1=V_0$. The functor $V\to\Pi V$ is a self-equivalence of categories $G-\mathsf{SMod}$ and $\mathsf{SMod}-G$ and a self-equivalence of categories $G-\mathsf{Smod}$ and $\mathsf{Smod}-G$ as well. In what follows we shall adhere to Sweedler's notations, i.e. $\tau_V(v)=\sum v_1\otimes f_2$ for $v\in V$. 

Let $K[G]_r$ and $\rho_r$ denote the left $G$-supermodule structure on $K[G]$ induced by the right multiplication, i.e. $\rho_r=\Delta_{K[G]}$. Symmetrically, let $K[G]_l$ and $\rho_l$ denote the left $G$-supermodule structure on $K[G]$ induced by the left multiplication, i.e. $\rho_l(f)=\sum (-1)^{|f_1||f_2|} f_2\otimes s_{K[G]}(f_1)$, where $\Delta_{K[G]}(f)=\sum f_1\otimes f_2, f\in K[G]$. Then $s_{K[G]}$ maps isomorphically $K[G]_r$ onto $K[G]_l$. 

Let $X$ be a supergroup on which $G$ acts by supergroup automorphisms on the right (cf. \cite{zubgrish}, \S 6).
This is equivalent to $K[X]$ being a right $K[G]$-supercomodule such that its (super)comodule map $\tau_{[X]} : K[X]\to K[X]\otimes K[G]$, given by 
$f\mapsto\sum f_{(1)}\otimes f_{(2)}$, is a superalgebra morphism that satisfies the following two identities:
\begin{equation}\label{1}
\sum (-1)^{|(f_1)_{(2)}||(f_2)_{(1)}|}(f_1)_{(1)}\otimes (f_2)_{(1)}\otimes (f_1)_{(2)}(f_2)_{(2)}=\sum (f_{(1)})_1\otimes (f_{(1)})_2\otimes f_{(2)}
\end{equation}
and
\[\sum\epsilon_{K[X]}(f_{(1)})f_{(2)}=\epsilon_{K[X]}(f).\]

Since this action preserves the unit element, one can define a $G$-supermodule structure on the superspace $Dist(X)$ (see \cite{zubfr}, Lemma 5.5). 
Moreover, if we define a pairing
\[(Dist(X)\otimes A)\times (K[X]\otimes A)\to A\] by the rule 
\[(\phi\otimes a)(f\otimes b)=(-1)^{|a||f|}\phi(f)ab\]
for $\phi\in Dist(X), f\in K[X]$ and $a, b\in A$, then 
\[g\cdot(\phi\otimes a)(f\otimes b)=(\phi\otimes a)(g^{-1}(f\otimes b))\]
for $g\in G(A)$. 
\begin{lm}\label{inducedactiononDist}
For every $A\in\mathsf{SAlg}_K$ the group $G(A)$ acts on $Dist(X)\otimes A$ by superalgebra automorphisms. 
\end{lm}
\begin{proof}
One can define an $A$-supercoalgebra structure on $K[X]\otimes A$ by 
\[\Delta_{K[X]\otimes A}(f\otimes b)=\sum f_1\otimes 1\otimes f_2\otimes b \text{ and } \epsilon_{K[X]\otimes A}(f\otimes b)=\epsilon_{K[X]}(f)b.\]
Then, for arbitrary elements $\phi, \psi\in Dist(X)\otimes A$ and $f\in K[X]\otimes A$ there is the following formula
\[\phi\psi(f)=\sum(-1)^{|\psi||f_1|}\phi(f_1)\psi(f_2),\]
where as above, $\Delta_{K[X]\otimes A}(f)=\sum f_1\otimes f_2$. 
Using this formula and the identity $(\ref{1})$ for $\tau_{K[X]}$ one can prove the lemma by straightforward computations. 
\end{proof}
Let $Ad$ denote the left adjoint action of $G$ on $Dist(G)$, induced by the right conjugation action of $G$ on itself (cf. \cite{zubfr}, Lemma 5.5).
\begin{lm}\label{apropertyofadj}
For every $G$-supermodule $N$ we have \[x(\phi n)=(Ad(x)\phi)(xn),\] where $x\in G(A), \phi\in Dist(G), n\in N$ and $A\in\mathsf{SAlg}_K$.
\end{lm}
\begin{proof}
Using the following formulas 
\[x(\phi n)=\sum(-1)^{|\phi|(|n_1|+|f_2|)}n_1\otimes x(f_2)\phi(f_3)\]
and
\[Ad(x)\phi(r)=\sum (-1)^{|f_1||f_2|}\phi(f_2)x(f_1 s_G(f_3)).
\]
we derive
\[(Ad(x)\phi)(xn)=\sum (-1)^{|\phi||n_1| +|f_2||f_3|} n_1\otimes\phi(f_3)x(f_2 s_R(f_4) f_5)=x(\phi n).\]
\end{proof}
Let $R$ and $L$ be supersubgroups of $G$ such that $R$ normalizes $L$. By Lemma \ref{inducedactiononDist}, $R$ acts on $Dist(L)$ by superalgebra automorphisms via $Ad|_R$.
\begin{lm}\label{alifting}
Let $V$ be a $R$-supermodule and a $L$-supermodule, and these two structures be related in such a way that 
$x(\phi v)=(Ad(x)\phi)(xv)$, where $x\in R(A), \phi\in Dist(L), v\in V$ and $A\in\mathsf{SAlg}_K$.
If $L$ is connected, then for every $y\in L(A)$ we have $x(yv)=(xyx^{-1})(xv)$, where $xyx^{-1}$ is computed in $G(A)$.
\end{lm}
\begin{proof}
The structure of the $R$-supermodule on $V$ (and the $L$-supermodule on $V$, respectively) is given by the supercomodule map
$v\mapsto\sum v_1\otimes v_2$ (and $v\mapsto\sum v_{(1)}\otimes v_{(2)}$, respectively).
The first condition of the lemma is equivalent to the identity
\[\sum (-1)^{|v_{(1)}||\phi|}(v_{(1)})_1\otimes x((v_{(1)})_2)\phi(v_{(2)})=\]
\[\sum (-1)^{|\phi||(v_1)_{(1)}|+|(v_1)_{(2)}||(v_1)_{(3)}|} (v_1)_{(1)}\otimes \phi((v_1)_{(3)})
h((v_1)_{(2)})h^{-1}((v_1)_{(4)})h(v_2),
\]
valid for every element $\phi\in Dist(L)$. 
Denote by $C$ the expression
\[\sum (v_{(1)})_1\otimes x((v_{(1)})_2)\otimes v_{(2)}\]
and by $D$ the expression
\[\sum (-1)^{|\phi|(|(v_1)_{(4)}|+|v_2|)} (v_1)_{(1)}\otimes h((v_1)_{(2)})h^{-1}((v_1)_{(4)})h(v_2)\otimes (v_1)_{(3)}.
\]
If $L$ is connected, then $\cap_{n\geq 1}(K[L]^+)^n =0$ (cf. \cite{zubgrish}). 
Using $(id_V\otimes id_A\otimes \phi)(v\otimes a \otimes f)=(-1)^{|\phi|(|a|+|f|)}v\otimes a\otimes\phi(f)$
we derive
\[C-D\in\cap_{\phi\in Dist(L)}\ker (id_V\otimes id_A\otimes \phi)\subseteq V\otimes A\otimes\cap_{n\geq 1}(K[L]^+)^n =0,\]
and therefore $C=D$. 
After we apply $id_V\otimes id_A\otimes y$, where $y\in L(A)$, we obtain the second identity. 
\end{proof}
There is an equivalence of categories $G-\mathsf{SMod}\simeq \mathsf{SMod}-G$ given by $V\mapsto V^{\circ}$, where $V^{\circ}$ coincides with $V$ as a superspace and the structure of a left $K[G]$-supercomodule on $V^{\circ}$ is defined by \[\tau_{V^{\circ}}(v)=\sum (-1)^{|v_1||f_2|} s_{K[G]}(f_2)\otimes v_1 ,\] provided $\tau_V(v)=\sum v_1\otimes f_2$ for $v\in V$. In other words, an element $g\in G(A)$ for $A\in\mathsf{SAlg}_K$ acts on $V^{\circ}$ by $(1\otimes v)g=\sum g^{-1}(f_2)\otimes v_1$ for $v\in V$. 

As a special case, we have \[\rho_r^{\circ}(f)=\sum (-1)^{|f_1||f_2|} s_{K[G]}(f_2)\otimes f_1\] for $f\in K[G]$ and $\rho_l^{\circ}=\Delta_{K[G]}$. As above, $s_{K[G]}$ maps isomorphically $K[G]^{\circ}_r$ onto $K[G]_l^{\circ}$.

Every right $G$-supermodule $V$ is also a right $Dist(G)$-supermodule via $v\phi=\sum (-1)^{|\phi||v_1|}\phi(f_2)v_1$, provided $\tau_V (v)=\sum f_2\otimes v_1$ for $v\in V$.

For every $V\in \mathsf{SMod}-G$ and $A\in\mathsf{SAlg}_K$ one can define the pairing
\[(A\otimes V)\times (A\otimes V^*)\to A\]
by
\[(a\otimes v)(b\otimes\phi)\mapsto (-1)^{|v||b|}ab(v)\phi,\] 
where $a, b\in A, v\in V$ and $\phi\in V^*$.
Then
\[(v')(\phi' g)=(v'g^{-1})\phi',\]
for $v'\in A\otimes V, \phi'\in A\otimes V^*$ and $g\in G(A)$.

If $W$ is a right $Dist(G)$-supermodule, then $A\otimes W$ is a right $A\otimes Dist(G)$-supermodule via 
\[(a\otimes w)(b\otimes\psi)=(-1)^{|w||b|}ab\otimes w\psi.\]
Note that an analogous statement is valid if we replace $Dist(G)$ by any associative superalgebra.
Then as above, we have
\[(v')(\phi'\psi')=(-1)^{|\phi'||\psi'|}(v' s_{Dist(G)}(\psi'))\phi',\]
where $\psi'\in A\otimes Dist(G)$ and $s_{Dist(G)}(b\otimes\psi)=b\otimes s_{Dist(G)}(\psi)$. 
The left-hand-side counterparts of these statements have been proved in \cite{zubfr}. 
\begin{rem}
If $V$ is a finite-dimensional left/right $G$-supermodule, then the natural isomophism $V\to (V^*)^*$ is given by $v\mapsto\overline{v}$, where $\overline{v}(\phi)=(-1)^{|v|}\phi(v)$, 
$v\in V$ and $\phi\in V^*$.\end{rem}
Let $V$ be a $G$-supermodule and $L$ be an irreducible $G$-supermodule. Denote by $[V : L]$ the sum of multiplicities of $L$ and $\Pi L$ in the composition series of $V$.

Two irreducible $G$-supermodules $L$ and $L'$ are linked if and only if there is a chain of irreducible $G$-supermodules $L=L_0, L_1, \ldots , L_t=L'$ such that for each $0\leq i\leq t-1$ either $Ext^1_G(L_i, L_{i+1})\neq 0$ or $Ext^1_G(L_{i+1}, L_i)\neq 0$. A block $B$ of $G$ consists of all irreducible $G$-supermodules that are linked. 
For all $G$-supermodules $V$ and $W$ and every $i\geq 0$ and $a, b\in\mathbb{Z}_2$ there is an isomorphism (cf. \cite{zub4})
\[Ext^i_G(\Pi^a V, \Pi^b W)\simeq \Pi^{a+b} Ext^i_G(V, W).\]
Therefore every block is invariant under the parity shift functor.

We say that a supermodule $M$ belongs to the block $B$ if all composition factors of $M$ belong to $B$. 
Using this definition, we will show that if $G$-supermodules $M$ and $N$ belong to different blocks, then $Hom_G(M, N)=0$. 
If we assume that $Hom_G(M, N)\neq 0$, then either $Hom_{G-\mathsf{Smod}}(M, N)=Hom_G(M, N)_0\neq 0$ or 
$Hom_{G-\mathsf{Smod}}(\Pi M, N)=\Pi Hom_G(M, N)_1\neq 0$ showing that $M$ and $N$ belong to the same block.
The analogous statement is clearly valid for right $G$-supermodules. 

Standard arguments (cf. \cite{jan}, II.7.1) show that every $G$-supermodule $N$ can be decomposed as 
$N=\oplus N_{B}$, where the summation is over all blocks $B$ and each $N_{B}$ is the largest $G$-supersubmodule of $N$ that belongs to $B$. 
Since there are no morphisms or higher extensions between supermodules belonging to different blocks,  the category $G-\mathsf{SMod}$ is a direct sum of its full subcategories $G-\mathsf{SMod}_{B}$ for different blocks $B$. For example, if two irreducible supermodules $L$ and $L'$ are 
composition factors of an indecomposable supermodule $V$, then they belong to the same block.

\section{Finite supergroups}

A partial case of the MacKay imprimitivity theorem from \cite{zub1} states the following. 
Let $G'$ be an algebraic supergroup, $G$ be a normal supersubgroup of $G'$, $H'$ be a supersubgroup of $G'$ and $H=G\cap H'$. Then 
$(ind^{GH'}_{H'} M)|_G \simeq ind^G_H M|_H$ as $G$-supermodules.
Using this isomorphism, we can lift $G$-supermodule structure on $ind^G_H M|_H$ to a $GH'$-supermodule structure.

In this section, for infinitesimal supergroup $G$, we lift the $G$-supermodule structure on $coind^G_H M|_H$ to a $GH'$-supermodule structure in such a way that 
\[(coind^{GH'}_{H'} M)|_G \simeq coind^G_H M|_H\] as $G$-supermodules. Also, we prove results about duals of induced and coinduced supermodules which will be utilized later. 

Let $H$ be an infinitesimal supergroup. It has been shown in \cite{zubfr} that the supersubspace $\int_{{\bf r}, Dist(H)}$ of $Dist(H)=K[H]^*$,
consisting of all right integrals on $Dist(H)$, is one-dimensional. 
It follows from Lemma 5.5 of \cite{zubfr} that if $H$ is a normal supersubgroup of an algebraic supergroup $R$, then
$\int_{{\bf r}, Dist(H)}$ is a $R$-supersubmodule of $Dist(R)$ with respect to the adjoint action of $R$ on $Dist(R)$. 
Therefore, $R$ acts on $\int_{{\bf r}, Dist(H)}$ via a character $\chi_r\in X(R)$.
Analogously, $R$ acts on the superspace of left integrals $\int_{{\bf l}, Dist(H)}$ via a character $\chi_l$.

\begin{rem}\label{coincidenceofcharacters}
If $H$ is infinitesimal, then $\chi_r=\chi_l^{-1}$ and the parities of $\int_{{\bf l}, Dist(H)}$ and $\int_{{\bf r}, Dist(H)}$ are the same. Since $Ad(h)(s_{Dist(H)}(\mu))=s_{Dist(H)}(Ad(h^{-1})\mu)$ for every $\mu\in Dist(G)$, this follows from Remark 2.4 of \cite{zubfr}.
\end{rem}
\begin{lm}\label{anisomforright}
If $H$ is infinitesimal, then the right $H$-supermodule $(K[H]_r^{\circ})^*$ is isomorphic to $\Pi^a K[H]_r^{\circ}\otimes\chi_l=\Pi^a K[H]_r^{\circ}\otimes\chi_r^{-1}$, 
where $a$ is the parity of $\int_{{\bf l}, Dist(H)}$.
Analogously, $(K[H]_l^{\circ})^*\simeq\Pi^a K[H]_l^{\circ}\otimes\chi_r= \Pi^a K[H]_l^{\circ}\otimes\chi_l^{-1}$.
\end{lm}
\begin{proof}
Denote by $Int$ the left action of $R$ on $K[R]$, induced by the right conjugation action of $R$ on itself.
By Remark 2.4 of \cite{zubfr}, there is an isomorphism $\Pi^{|\nu|}K[H]\to K[H]^*=Dist(H)$ of left $H$-supermodules, given by 
$f\mapsto \nu f$, where $\nu\in \int_{{\bf l}, Dist(H)}\setminus 0$. 
Recall that $\nu\in (K[H]_l^*)^H$ and $(\nu f)(f')=\nu(f f')$. 
For the sake of simplicity, for $h\in H$ and $f\in K[H]$, denote $\rho_l(h)f$ and $f\rho_r^{\circ}(h)$ by $hf$ and $fh$, respectively. 
Then
\[((\nu f)h)(f')=\nu(f(f' h^{-1}))=\nu(Int(h)(Int(h^{-1})(f) h^{-1}f'))=\]
\[(Ad(h^{-1})\nu)(Int(h^{-1})(f) h^{-1}f')=\nu(Int(h^{-1})(f) h^{-1}f'))\chi_l^{-1}(h)=
\]
\[\nu(h^{-1}((fh)f'))\chi_l^{-1}(h)=(h\nu)((fh)f')\chi_l^{-1}(h)=\nu(fh) (f')\chi_l^{-1}(h),\]
which proves $(K[H]_r^{\circ})^*\simeq \Pi^a K[H]_r^{\circ}\otimes\chi_l$.
The proof of the second statement is analogous.
\end{proof}
\begin{rem}\label{interpretation}
Let $G'$ be an algebraic supergroup, $G$ be its normal supersubgroup, $H'$ be a supersubgroup of $G'$ and $H=H'\cap G$. 
According to Theorem 10.1 of \cite{zub1}, for every $H'$-supermodule $M$ there is a superspace isomorphism $ind^{GH'}_{H'} M\to ind^G_H M|_H$ 
that naturally induces a $GH'$-supermodule structure on $ind^G_H M|_H$ such that its restriction to $G$ induces an isomorphism 
$(ind^{GH'}_{H'} M)|_G\to ind^G_H M|_H$ of $G$-supermodules. 
More precisely, one can  interpret the $G$-supermodule $ind^G_H M|_H$ as a $G$-subfunctor of $\mathfrak{Mor}(G, M_a)$ (cf. \cite{zub1}, \S 8), consisting of all affine superscheme morphisms $\mathfrak{f} : G\to N_a$ such that
$\mathfrak{f}(gh)=h^{-1}\mathfrak{f}(g)$ for $g\in G, h\in H$. A $G$-action on $\mathfrak{Mor}(G, N_a)$ is given by $g\mathfrak{f}(g')=\mathfrak{f}(g^{-1}g')$. Then the above $H'$-action on $ind^G_H M|_H$ is defined by
$(h'*\mathfrak{f})(g)=h'\mathfrak{f}(h'^{-1}gh')$.
\end{rem}
\begin{lm}\label{adecompofdist}
Let $G'$ be an algebraic supergroup, $G$ be its normal supersubgroup, and $H'$ be a supersubgroup of $G'$. Then
$Dist(GH')=Dist(G)Dist(H')$.
\end{lm}
\begin{proof}
It has been observed in \S 7 of \cite{zub3} that $GH'$ is an epimorphic image of the supergroup
$G\ltimes H'$. By the proof of Proposition 9.1 of \cite{zub3}, there is  a superalgebra epimorphism
$Dist(G\ltimes H')\to Dist(GH')$. Since $Dist(G\ltimes H')=Dist(G)\otimes Dist(H')$, the statement follows from results in \S 4 of \cite{zub1}.
\end{proof}

From now on until the end of this section, we assume that $G$ is an infinitesimal supergroup. 

By Lemma 1.3 from \cite{zub2}, $G-\mathsf{SMod}$ can be naturally identified with the category of left $Dist(G)$-supermodules $_{Dist(G)}\mathsf{SMod}$. For example, $Dist(G)$-supermodule $coind^G_H M=Dist(G)\otimes_{Dist(H)} M$ has a unique structure of a $G$-supermodule such that its induced $Dist(G)$-supermodule structure coincides with the original one. 
Moreover, for every $M, N\in G-\mathsf{SMod}$ (or $\mathsf{SMod}-G$, respectively) we have $Hom_G(M, N)=Hom_{Dist(G)}(M, N)$.
\begin{rem}\label{explanation}
Let $A$ be an associative (not necessary supercommutative) superalgebra. 
For all (left) $A$-supermodules $M$ and $N$ the superspace $Hom_A(M, N)$ is generated by all linear homogeneous maps
$\phi : M\to N$ such that $\phi(am)=(-1)^{|a||\phi|}a\phi(m)$ for $a\in A$ and $m\in M$. 
We would like to take an opportunity to clarify Lemma 4.4 from \cite{zub2} and replace the 
incorrect statement $_{\mathsf{Dist}(H)}\mathsf{SMod}\simeq _{\mathsf{Dist}(H)\rtimes\mathbb{Z}_2}\mathsf{Mod}$ with 
 $_{\mathsf{Dist}(H)}\mathsf{Smod}\simeq _{\mathsf{Dist}(H)\rtimes\mathbb{Z}_2}\mathsf{Mod}$.
\end{rem}
The left $Dist(G)$-supermodule $(K[G]_l)^*$ is isomorphic to $Dist(G)$ as a left $Dist(G)$-supermodule and the right $Dist(G)$-supermodule $(K[G]_r^{\circ})^*$ is isomorphic to $Dist(G)$ as a right $Dist(G)$-supermodule. Since Theorem 0.1 of \cite{zub5} and Theorem 5.2(4) of  \cite{zub3} imply that both $K[G]_l$ and $K[G]_r^{\circ}$ are injective $H$-supermodules, 
$Dist(G)$ is a projective left and right $Dist(H)$-supermodule.  

Let $A$ and $B$ be (not necessarily finite-dimensional) associative superalgebras. 
\begin{lm}\label{anidentity}
For arbitrary left $B$-supermodule $M$ and right projective $B$-supermodule $M'$ there is a natural superspace isomorphism 
\[M'\otimes_B M\simeq Hom_{\mathsf{_{B}SMod}} (Hom_{\mathsf{SMod}_B} (M', B), M). \]
If $M'$ is also a left $A$-supermodule, that is $M'$ is an $A\times B$-superbimodule, then the above map is an isomorphism of $A$-supermodules. 
\end{lm}
\begin{proof}
The required isomorphism sends $m'\otimes m$ to the map $\phi\mapsto (-1)^{|\phi||m|}((m')\phi) m$ for $\phi\in Hom_{\mathsf{SMod}_B} (M', B)$ (cf. I.8.16(5), \cite{jan}). 
This assignment is obviously a superspace morphism. Since it commutes with direct sums, to show that it is an isomorphism, all we need is to check the case $M'=B$, which is clear.
 
If $M'$ is an $A\times B$-superbimodule, then $Hom_{\mathsf{SMod}_B} (M', B)$ is a $B\times A$-superbimodule 
via \[(m')(b\phi a)=(-1)^{|b||m'|+|a|(|\phi|+|m'|)}b((am')\phi)\] for $a\in A, b\in B$. 
The right $A$-supermodule structure on $Hom_{\mathsf{SMod}_B} (M', B)$ induces a natural left $A$-supermodule structure on   
$Hom_{\mathsf{_{B}SMod}} (Hom_{\mathsf{SMod}_B} (M', B), M)$, which is obviously compatible with the $A$-supermodule structure on $M'\otimes_B M$ given by our isomorphism.
\end{proof}
\begin{cor}\label{forDist} If we set $M'=Dist(G)$, then we obtain an isomorphism of $G$-supermodules
\[coind^G_H M\simeq Hom_{H-\mathsf{SMod}} (Hom_{\mathsf{SMod}-H} (Dist(G), Dist(H)), M).\] 
\end{cor}

For the remainder of this section assume that $G'$ be an algebraic supergroup, $G$ be its normal supersubgroup, $H'$ be a supersubgroup of $G'$ and $H=H'\cap G$.
Let $G'$ act on $\int_{{\bf l}, Dist(G)}$ via the character $\alpha'$ and denote $\alpha=\alpha'|_G$. Also, let $H'$ act on $\int_{{\bf r}, Dist(H)}$ via a characer $\chi'$ and 
denote $\chi=\chi'|_H$. Further, denote by $a$ the parity of $\int_{{\bf l}, Dist(H)}$ and by $b$ the parity of $\int_{{\bf l}, Dist(G)}$.

The following proposition superizes Proposition I.8.17 of \cite{jan}. 
\begin{pr}\label{coind and ind}
For every $H$-supermodule $M$ there is an isomorphism 
\[coind^G_H M\simeq ind^G_H(\Pi^{a+b} M\otimes\alpha|_H\chi^{-1}).\]
\end{pr}
\begin{proof}
The superspace $V=Hom_{\mathsf{SMod}-H} (Dist(G), Dist(H))$ is isomorphic to
$(K[G]_r^{\circ}\otimes Dist(H))^H$, where an element $f\otimes\phi\in K[G]\otimes Dist(H)$ acts on $Dist(G)$ as
\[\psi(f\otimes\phi)=(-1)^{|f|}\psi(f)\phi\]
for $\psi\in Dist(G)$.  
Moreover, $Hom_K(Dist(G), Dist(H))$ and $K[G]_l^{\circ}\otimes Dist(H)_{triv}$ are isomorphic as right $G$-supermodules. Symmetrically, $Hom_K(Dist(G), Dist(H))$ and
$K[G]_{triv}\otimes (K[H]_l)^*\simeq K[G]_{triv}\otimes K[H]_l$ are isomorphic as left $H$-supermodules.
 
Lemma \ref{anisomforright} implies that $Dist(H)\simeq \Pi^{a}K[H]_r^{\circ}\otimes\chi^{-1}$ as right $H$-supermodules. Therefore $V$ is isomorphic to
\[^{H}_{H}ind (\Pi^a K[G]^{\circ}_r\otimes\chi^{-1})\simeq\Pi^a (K[G]_r^{\circ}\otimes ^{H}_{H}ind (\chi^{-1}))=\Pi^a K[G]^{\circ}_r\otimes\chi^{-1}, \]
where $^{H}_{H}ind$ is a right-hand-side counterpart of the functor $ind^H_H$. This isomorphism is defined as 
\[f\otimes\chi^{-1}\mapsto \sum f_1\otimes\chi^{-1}\otimes\overline{s_{K[G]}(f_2)}\chi^{-1},\]
where $\overline{s_{K[G]}(f_2)}$ is the image of $s_{K[G]}(f_2)$ in $K[H]$. In particular, $V$ is isomorphic to $\Pi^a K[G]_l^{\circ}$ as a right $G$-supermodule, 
and to $\Pi^a K[G]_r\otimes\chi$ as a left $H$-supermodule.
 
Consequently, $coind^G_H M$ is isomorphic to 
\[((\Pi^a(K[G]_r\otimes\chi)^*\otimes M)^H\simeq \Pi^a(K[G]_r^*\otimes\chi^{-1}\otimes M)^H.\] 
On the other hand, $coind^G_H M$ as a $G$-supermodule is isomorphic to $(K[G]^{\circ}_l)^*\otimes M_{triv}$, where $G$ acts on $(K[G]_l^{\circ})^*$ 
via anti-automorphism $g\mapsto g^{-1}$ for $g\in G$. This together with our Lemma \ref{anisomforright} and Corollary 2.3 of  \cite{zubfr} concludes the argument. 
\end{proof}
\begin{lm}\label{doubleduality}
If $M$ is a finite-dimensional $H$-supermodule, then $ind^G_H (M^*)$ is naturally isomorphic to $(coind^G_H M)^*$.
\end{lm}
\begin{proof}
Modify the proof of Lemma I.8.15 of \cite{jan} using Proposition 2.1 of \cite{zub4}.
\end{proof}
\begin{pr}\label{alifttoGH'}
If $M$ is a $H'$-supermodule, then the $Dist(GH')$-supermodule structure on $Dist(GH')\otimes_{Dist(H')} M$ can be lifted to the structure of a $GH'$-supermodule.
\end{pr}
\begin{proof}
There is a structure of a $H'$-supermodule on $Dist(GH')\otimes M$ given by 
\[x(\phi\otimes m)=Ad(x)\phi\otimes xm\]
for $x\in H'(A)$ and $A\in\mathsf{SAlg}_K$.
Because
\[x(\phi\psi\otimes m-\phi\otimes\psi m)=Ad(x)\phi Ad(x)\psi\otimes xm-Ad(x)\phi\otimes x(\psi m)=\]
\[Ad(x)\phi Ad(x)\psi\otimes xm-Ad(x)\phi\otimes (Ad(x)\psi)(xm),\]
this provides a structure of an $H'$-supermodule on $V=Dist(GH')\otimes_{Dist(H')} M$.

We have already observed that a structure of a $Dist(G)$-supermodule lifts to a structure of a $G$-supermodule. 
Therefore, $V$ is a $G$-supermodule via lifting of its left $Dist(G)$-supermodule structure. 
Lemma \ref{inducedactiononDist} implies $x(\phi v)=Ad(x)(\phi) xv$ for $\phi\in Dist(G)$ and $v\in V$.  Using this and Lemma \ref{alifting} we infer that $V$ is a $G\ltimes H'$-supermodule. 

Accoding to Lemma 11.1 of \cite{zub2} and Lemma 5.5 of \cite{zubfr}, $Dist(H')$ acts on $V$ by the following rule 
\[\phi(\psi\otimes m)=\sum (-1)^{|\phi_2||\psi|} ad(\phi_1)\psi\otimes \phi_2 m,\]
where $\Delta_{Dist(H')}(\phi)=\sum \phi_1\otimes\phi_2, \phi\in Dist(H')$. Thus
\[\phi(\psi\otimes m)=\sum (-1)^{|\psi|(|\phi_2|+|\phi_3|)}\phi_1\psi s_{Dist(H')}(\phi_2)\phi_3\otimes m=\phi\psi\otimes m.
\]
Finally, $GH'$ is naturally isomorphic to $(G\ltimes H')/H$, where $H$ is embedded into $G\ltimes H'$ via  the map $x\mapsto (x, x^{-1})$ for $x\in H$. 
Consequently, $Dist(H)$ is embedded into $Dist(G\ltimes H')\simeq Dist(G)\otimes Dist(H')$ via the map $\phi\mapsto \sum\phi_1\otimes s_{Dist(H)}(\phi_2)$.
Using the above observations and Lemma 9.5 of \cite{zub3}, we see that $H\leq G\ltimes H'$ acts on $V$ trivially and our claim follows.
\end{proof}
There is a natural $Dist(G)$-supermodule epimorpism 
\[\tau: coind^G_H M|_{H}=Dist(G)\otimes_{Dist(H)} M|_H\to Dist(GH')\otimes_{Dist(H')} M.\]
If $\tau$ is an isomorphism, then the $G$-supermodule structure on $coind^G_H M|_H$ can be extended to a structure of a $GH'$-supermodule.
In this case, following \cite{jan}, we denote this $GH'$-supermodule $coind^G_H M|_H$ by $coind^{GH'}_{H'} M$.

Untill the end of this section we assume the supermodule $M$ is such that the $G$-supermodule structure on $coind^G_H M|_H$ can be lifted to 
the $GH'$-supermodule structure as above.

\begin{rem}\label{asuffcond}
There is a simple sufficient condition for the natural epimorphism $\tau$ to be an isomorphism.
Assume that $G$ contains a closed supersubscheme (not necessary supersubgroup) $W$ such that 
$W\cap H=1$ and the multiplication map $W\times H\to G$ is an isomorphism of superschemes. Then 
$\tau$ is an isomorphism because $GH'\simeq W\times H'$ and both $Dist(G)$ and $Dist(GH')$ are free right supermodules over $Dist(H)$ and $Dist(H')$, respectively. 
More precisely, $Dist(G)\simeq Dist(W, \mathfrak{m}_1)\otimes Dist(H)$ and $Dist(GH')\simeq Dist(W, \mathfrak{m}_1)\otimes Dist(H')$.
\end{rem}

\begin{pr}\label{thesameforGH'} Assume $M$ satisfies the above assumption. Then
\[coind^{GH'}_{H'} M\simeq ind^{GH'}_{H'} \Pi^{a+b} M\otimes\alpha'|_{H'}\chi'^{-1}.\]
Additionally, if $M$ is finite-dimensional, then 
\[(ind^{GH'}_{H'} M)^*\simeq ind^{GH'}_{H'} \Pi^{a+b} M^* \otimes\alpha'|_{H'}\chi'.\]
\end{pr}
\begin{proof}
For $V\in H'-\mathsf{SMod}$, the action of $H'$ on $coind^G_H V|_H$ is given by the rule 
\[x(\phi\otimes m)=Ad(x)\phi\otimes xm\]
for $x\in H'(A), \phi\in Dist(G)$ and $m\in V$. 
We need to check that this action commutes with the isomorphism
\[coind^G_H V|_H\simeq ind^G_H \Pi^{a+b} V\otimes\alpha|_H\chi^{-1}\simeq ind^{GH'}_{H'} \Pi^{a+b} V\otimes\alpha'|_{H'}\chi'^{-1}.\]

We can use the isomorphism from Corollary \ref{forDist} to extend the action of $H'$ on $coind^G_H M$ to $H'$-action on 
$Hom_{H-\mathsf{SMod}} (Hom_{\mathsf{SMod}-H} (Dist(G), Dist(H)), M)$ by 
\[(x\psi)(\pi)=x(\psi(\pi x)), (\phi)(\pi x)=Ad(x^{-1})((Ad(x)\phi)(\pi)),\] where 
\[\psi\in Hom_{H-\mathsf{SMod}} (Hom_{\mathsf{SMod}-H} (Dist(G), Dist(H)), M),\]
\[\pi\in Hom_{\mathsf{SMod}-H} (Dist(G), Dist(H)), \text{ and } \phi\in Dist(G).\]
Therefore $H'$ acts on $Hom_K (Dist(G), Dist(H))\simeq K[G]\otimes Dist(H)$ as
$(f\otimes\mu)x=Int(x^{-1})f\otimes Ad(x^{-1})\mu$ for $f\in K[G]$ and $\mu\in Dist(H)$. 
This action obviously commutes with the right $H$-action and preserves the $H\times G$-superbimodule structure of $Hom_K (Dist(G), Dist(H))$. 
Thus 
\[Hom_{\mathsf{SMod}-H} (Dist(G), Dist(H))\simeq\Pi^a K[G]\otimes\chi^{-1}, \]
where $H'$ acts on $K[G]$ by $fx=Int(x^{-1})f$. Arguing as in Proposition \ref{coind and ind} we infer that the $H'$-action on
\[coind^G_H M\simeq (K[G]_r\otimes M\otimes \alpha|_H\chi^{-1})^H\]
is induced by $Int$ on $K[G]$, and on the remaining tensor factors it is induced by the action on $M\otimes\alpha'|_{H'}\chi'^{-1}$. 
Remark \ref{interpretation} concludes the proof of the first statement of the Lemma.

If $M$ is a finite-dimensional $H'$-supermodule, then a generalization of Lemma \ref{doubleduality} gives a natural isomorphism 
\[coind^{GH'}_{H'} M^*\simeq (ind^{GH'}_{H'} M)^*,\] showing that the second statement follows from the first one.
\end{proof}

\section{Borel supersubgroups and root systems}\label{3}

Let $G=GL(m|n)$. The coordinate superalgebra $K[G]$ is generated by the matrix coefficients $c_{ij}$ for $1\leq i, j\leq m+n$ such that its parity $|c_{ij}|=0$ if and only if 
$1\leq i, j\leq m$ or $m+1\leq i, j\leq m+n,$ and by the element $D^{-1}$, where $D=\det(C_{00})\det(C_{11})$, $C_{00}=(c_{ij})_{1\leq i, j\leq m}$ and $C_{11}=(c_{ij})_{m+1\leq i, j\leq m+n}$. 
Recall that $\Delta_{K[G]}(c_{ij})=\sum_{1\leq k\leq m+n} c_{ik}\otimes c_{kj}$ and $\epsilon_{K[G]}(c_{ij})=\delta_{ij}$. 
Denote the supersubalgebra of $K[G]$ generated by all elements $c_{ij}$ by $A(m|n)$. Then $K[G]=A(m|n)_D$.

Let $Ber$ denote the group-like element $\det(C_{00}-C_{01}C_{11}^{-1}C_{10})\det(C_{11})^{-1}$, where $C_{01}=(c_{ij})_{1\leq i\leq m, m+1\leq j\leq m+n}$ and 
$C_{10}=(c_{ij})_{m+1\leq i\leq m+n, 1\leq j\leq m}$. The element $Ber$ is a character of $GL(m|n)$ which is called {\it Berezinian}. If $H$ is an algebraic supergroup and $V$ is a $H$-supermodule, then any supermodule structure on $V$ is uniquely defined by a supergroup morphism $f : H\to GL(V)=GL(m|n)$ where $m=\dim V_0, n=\dim V_1$. 
The character $Ber\circ f$ of $H$ will be denoted by $Ber(f)$.

We fix the standard maximal torus $T$ such that $T(A)$ consists of all diagonal 
matrices from $G(A)$ for $A\in\mathsf{SAlg}_K$ and identify $X(T)$ with the additive group $\mathbb{Z}^{m+n}$. 
In particular, every $\lambda\in X(T)$ has the form
$\sum_{1\leq i\leq m+n}\lambda_i\epsilon_i$, where 
\[\epsilon_i(t)=t_i, t=\left(\begin{array}{cccc}
t_1 & 0 & \ldots & 0 \\
0   & t_2 & \ldots & 0 \\
\vdots & \vdots & \ldots & 0 \\
0 & 0 & \dots & t_n
\end{array}\right)\in T(A),\] 
$A\in\mathsf{SAlg}_K$ and every $\lambda_i$ is an integer. 
For a character $\lambda\in X(T)$ as above we denote $\sum_{1\leq i\leq m+n}\lambda_i$ by $|\lambda|$. 

The bilinear form on $X(T)\otimes_{\mathbb{Z}}\mathbb{Q}$ is defined by $(\epsilon_i, \epsilon_j)=(-1)^{|\epsilon_i|}\delta_{ij}$, where $|\epsilon_i|=0$ if $1\leq i\leq m$
and $|\epsilon_i|=1$ if $m+1\leq i \leq m+n$. If we denote $\epsilon'_i=(-1)^{|\epsilon_i|}\epsilon_i$, 
then $(\epsilon_i, \epsilon'_j)=\delta_{ij}$.

The root system of $G$ is defined as $\Phi =\{\epsilon_i-\epsilon_j| 1\leq i\neq j\leq m+n\}$. Let $S_k$ denote the symmetric group on the elements $\{1, \ldots, k\}$. 
For a given $w\in S_{m+n}$ we have the decomposition $\Phi=\Phi^+_w\bigcup \Phi^-_w$, where 
$\Phi^+_w =\{\epsilon_{wi}-\epsilon_{wj}| 1\leq i < j\leq m+n\}$ is the $w$-{\it positive} part
of $\Phi$ and $\Phi^-_w =-\Phi_w^+=\{\epsilon_{wi}-\epsilon_{wj}| 1\leq i > j\leq m+n\}$ is the $w$-{\it negative} part of $\Phi$.  

The simple roots of $\Phi^+_w$ form a subset $$\Pi_w=\{\alpha_i=\epsilon_{wi}-\epsilon_{w(i+1)}| 1\leq i < m+n\}.$$
The set $\Phi^+_w$ defines a partial order $<_w$ on the weight lattice $X(T)$ by $\mu <_w\lambda$ if $\lambda-\mu\in\sum_{\alpha\in\Phi_w^+} \mathbb{N}_+\alpha=\sum_{\alpha\in\Pi_w} \mathbb{N}_+\alpha$.

Denote the parity $|\epsilon_i|+|\epsilon_j|=|c_{ij}|$ of the root $\alpha=\epsilon_i-\epsilon_j$ by $p(\alpha)$. The coroot $\alpha^{\vee}=(\epsilon_i-\epsilon_j)^{\vee}$ equals $\epsilon'_i-\epsilon'_j$. 
For any $\alpha\in \Phi$ one can define a reflection $s_{\alpha}$ such that
$s_{\alpha}(\lambda)=\lambda-(\lambda, \alpha^{\vee})\alpha$.
It is easy to see that if $\alpha=\epsilon_i-\epsilon_j$, then $s_{\alpha}$ corresponds to the transposition of the $i$-th and $j$-th entries and therefore $s_{\alpha}$ will be identified with $(ij)$ of $S_{m+n}$. 
If $p(\alpha)=0$ we call $s_{\alpha}$ an {\it even} reflection, otherwise it is called an {\it odd} reflection. All reflections generate the group $S_{m+n}$ and the 
even reflections generate the {\it Weyl subgroup} $S_m\times S_n\subseteq S_{m+n}$. 

Denote by $D_{m, n}$ the set of representatives of $S_m\times S_n{/}S_{m+n}$-cosets that have the minimal length. 
According to \cite{brunkuj} or \cite{jan}, II.1.5(4), $D_{m, n}$ consists of all $w\in S_{m+n}$ such that 
\[w^{-1}(1)<\ldots < w^{-1}(m) \text{ and } w^{-1}(m+1)<\ldots < w^{-1}(m+n).\]
Then every $w\in S_{m+n}$ has a unique decomposition $w=w_0 w_1$, where $w_0\in S_m\times S_n$ and $w_1\in D_{m, n}$. We call this the {\it regular} decomposition of $w$.

Let $\rho_0(w)$ denote $\frac{1}{2}\sum_{\alpha\in\Phi^+_w, p(\alpha)=0}\alpha$, $\rho_1(w)$ denote $\frac{1}{2}\sum_{\alpha\in\Phi^+_w , p(\alpha)=1}\alpha$
and $\rho$ denote $\rho_0(w)-\rho_1(w)$. The elements $\rho_0(1)$, $\rho_1(1)$ and $\rho(1)$, respectively will be denoted just by $\rho_0$, $\rho_1$ and $\rho$, respectively.

The Borel supersubgroup $B^+_w$ corresponding to $\Phi_w^+$ is the stabilizer of the full flag
\[V_1\subseteq V_2\subseteq\ldots\subseteq V_i\subseteq\ldots\subseteq V_{m+n}=V,\]
where $V_i=\sum_{1\leq s\leq i} Kv_{ws}$ for $1\leq i\leq m+n$. 
Symmetrically, the opposite Borel supersubgroup $B^-_w$ corresponding to $\Phi_w^-$ is the stabilizer of the full flag
\[W_1\subseteq W_2\subseteq\ldots\subseteq W_i\subseteq\ldots\subseteq W_{m+n}=V,\]
where $W_i=\sum_{m+n-i+1\leq s\leq m+n} Kv_{ws}$ for $1\leq i\leq m+n$. 

The unipotent radicals of $B^+_w$ and $B^-_w$, respectively are denoted by $U^+_w$ and $U^-_w$, respectively. 
We denote by $G_r$ the $r$-th Frobenius kernel of $G$. The $r$-th Frobenius kernels of the Borel subsupergroups and their unipotent radicals are denoted by 
$B^+_{r, w}$, $U^+_{r, w}$ and $B^-_{r, w}$, $U^-_{r, w}$, respectively. 

From now on, we will omit $w=1$ from all of the subscripts; for example, we will write $B^+$ and $<$ instead of $B^+_1$ and $<_1$.
Denote by $V^+$ (and $V^-$, respectively) the supersubgroup of $U^+$ (and $U^-$, respectively) consisting of all matrices defined by equations $c_{ij}=\delta_{ij}$ for $1\leq i,j\leq m$ 
and $m+1\leq i,j \leq m+n$. 

\section{Representations of $G_r B^{\pm}_w$ and $G_r T$}

In this section we discuss representation theory of Frobenius thickenings of certain supersubgroups of $G$. We derive results for compositions factors and formal characters of induced and coinduced supermodules.

Fix a positive integer $r$. For arbitrary $\lambda\in X(T)$ and $a\in\mathbb{Z}_2$ we define
$$Z'_{r, w}(\lambda^a)=ind^{G_r}_{B^-_{r, w}} K_{\lambda}^a, \ Z_{r, w}(\lambda^a)=coind^{G_r}_{B^+_{r, w}} K_{\lambda}^a=Dist(G_r)\otimes_{Dist(B_{r, w}^+)} K^a_{\lambda},$$
$$
\mathcal{Z}'_{r, w}(\lambda^a)=ind^{G_r T}_{B^-_{r, w} T} K_{\lambda}^a, \ \mathcal{Z}_{r, w}(\lambda^a)=coind^{G_r T}_{B^+_{r, w} T} K_{\lambda}^a=Dist(G_r T)\otimes_{Dist(B^+_{r, w} T)} K_{\lambda}^a,
$$
$$\Hat{Z}_{r, w}'(\lambda^a)=ind^{G_r B_w^-}_{B_w^-} K_{\lambda}^a , \ \Hat{Z}_{r, w}(\lambda^a)=coind^{G_r B_w^+}_{B_w^+} K^a_{\lambda} =Dist(G_r B_w^+)\otimes_{Dist(B_w^+)} K^a_{\lambda}.$$

\begin{lm}\label{decompositions}
There are superscheme isomorphisms
$$U^+_{r, w}\times B_w^- \simeq G_r B_w^- , \ U^-_{r, w}\times B_w^+ \simeq G_r B_w^+ ,$$
$$U^+_{r, w}\times B_{r, w}^- T\simeq G_r T , \ U^-_{r, w}\times B_{r, w}^+ T\simeq G_r T$$
and
$$U^+_{r, w}\times B_{r, w}^- \simeq G_r  , \ U^-_{r, w}\times B_{r, w}^+ \simeq G_r 
$$
induced by the multiplication map.
\end{lm}
\begin{proof}
By Remark \ref{asuffcond} it is enough to show that the last pair of morphisms are isomorphisms.
Using arguments of Remark 5.10 of \cite{zubfr} and the proof of Lemma 3.1 of \cite{zub3} we establish that 
for any supergroup from the list $\{U^{\pm}_{r, w}, B^{\pm}_{r, w}, G_r\}$ its superalgebra of distributions has a basis $\prod_{|f_i|=0}f_i^{(t_i)}\prod_{|f_i|=1}f_i^{s_i}$, where elements $f_i$ form a basis of its Lie superalgebra and their exponents satisfy $0\leq t_i\leq p^r -1$ and $0\leq s_i\leq 1$. The superalgebra multiplication induces a superspace isomorphism
\[Dist(U^{\pm}_{r, w})\otimes Dist(B^{\mp}_{r, w})\to Dist(G_r),\]
and the corresponding morphism $K[G_r]\to K[U^{\pm}_{r, w}]\otimes K[B^{\mp}_{r, w}]$ is also an isomorphism.
Since it is dual to the multiplication map $U^{\pm}_{r, w}\times B^{\mp}_{r, w}\to G_r$, the last map is an isomorphism (compare with Lemma 14.2 from \cite{zub1}). 
\end{proof}
\begin{lm}\label{soclesandtops}
For every $\lambda\in X(T)$ and $a\in\mathbb{Z}_2$, the socle of $\Hat{Z}'_{r, w}(\lambda^a)$ and the top of $\Hat{Z}_{r, w}(\lambda^a)$ are irreducible supermodules in 
$G_r B_w^-$-$\mathsf{SMod}$ or $G_r B_w^+$-$\mathsf{SMod}$, correspondingly. 

Analogous statements hold for $G_r T$-supermodules
$\mathcal{Z}'_{r, w}(\lambda^a), \mathcal{Z}_{r, w}(\lambda^a)$ and for $G_r$-supermodules
$Z'_{r, w}(\lambda^a), Z_{r, w}(\lambda^a)$. Moreover, in these cases the socle of the first supermodule and top of the second supermodule are isomorphic to each other.
\end{lm}
\begin{proof}
Combining our Lemma \ref{decompositions} with Lemma 8.2 and Remark 8.3 of \cite{zub1}, we infer that
$\Hat{Z}'_{r, w}(\lambda^a)|_{B^+_{r, w} T}\simeq K_{\lambda}^a\otimes K[U^+_{r, w}]$, where
$U_{r, w}^+$ acts trivially on $K_{\lambda}^a$ and $T$ acts on $K[U^+_{r, w}]$ by conjugations. 
Thus $\Hat{Z}'_{r, w}(\lambda^a)^{U^+_{r, w}}\simeq K_{\lambda}^a$ as a $T$-supermodule, hence the socle of $\Hat{Z}'_{r, w}(\lambda^a)$ is simple. 
Symmetrically, $\Hat{Z}_{r, w}(\lambda^a)\simeq Dist(U^-_{r, w})\otimes K_{\lambda}^a$ as a $B_{r, w}^- T$-supermodule, 
where $U^-_{r, w}$ acts trivially on $K_{\lambda}^a$ and $T$ acts on $Dist(U^-_{r, w})$ by conjugations. 
Thus $\Hat{Z}_{r, w}(\lambda^a)_{U^-_{r, w}}=\Hat{Z}_{r, w}(\lambda^a)/Dist(U^-_{r, w})^+\Hat{Z}_{r, w}(\lambda^a)\simeq K_{\lambda}^a$. 
According to \cite{zub2}, \S 4 the functor $M\to M_{U^-_{r, w}}$ is right exact, and therefore the top of $\Hat{Z}_{r, w}(\lambda^a)$ is simple.

Analogous arguments work for the induced/coinduced $G_r T$-supermodules and induced/coinduced $G_r$-supermodules as well.
Let $L'$ denote the socle of $\mathcal{Z}'(\lambda^a)$ and let $L$ denote the top of $\mathcal{Z}_{r, w}(\lambda^a)$. 
Then $L'^{U^+_{r, w}}\simeq K_{\lambda}^a$ as a $T$-supermodule. The (co)Frobenius reciprocity law implies
\[Hom_{G_r T}(\mathcal{Z}_{r, w}(\lambda^a), L')\simeq Hom_{B^+_{r, w} T}(K_{\lambda}^a , L')\simeq K,\] 
hence $L\simeq L'$. Similar arguments work in the case of $G_r$-supermodules.
\end{proof}
Denote by $\Hat{L}'_{r, w}(\lambda^a)$ the socle of $\Hat{Z}'_{r, w}(\lambda^a)$ and by $\Hat{L}_{r, w}(\lambda^a)$ the top of $\Hat{Z}_{r, w}(\lambda^a)$. 
Similarly, denote by $\mathcal{L}_{r, w}(\lambda^a)$ the socle of $\mathcal{Z}'_{r, w}(\lambda^a)$ (or the top of $\mathcal{Z}_{r, w}(\lambda^a)$) and 
by $L_{r, w}(\lambda^a)$ the socle of $Z'_{r, w}(\lambda^a)$ (or the top of $Z_{r, w}(\lambda^a)$). 
\begin{rem}\label{weights}
Let $M$ be a supermodule from the list 
\[\{\Hat{Z}'_{r, w}(\lambda^a), \Hat{Z}_{r, w}(\lambda^a), \mathcal{Z}'_{r, w}(\lambda^a), \mathcal{Z}_{r, w}(\lambda^a), Z'_{r, w}(\lambda^a), Z_{r, w}(\lambda^a)\}.\] 
The proof of Lemma \ref{soclesandtops} shows that if $M_{\mu}\neq 0$, then $\mu\leq_w\lambda$.
The same statement is valid for the simple supermodules $\Hat{L}'_{r, w}(\lambda^a), \Hat{L}_{r, w}(\lambda^a), \mathcal{L}_{r, w}(\lambda^a), L_{r, w}(\lambda^a)$ that are composition factors of such supermodules $M$.
\end{rem}
Let $H$ be an algebraic supergroup and $M$ be an $H$-supermodule. We call $M$ {\it cogenerated} by an element $m\in M$ if the socle of $M$ is irreducible and generated by $m$. Without loss of generality one can always assume that $m$ is homogeneous.

If $H$ is connected, then by Lemma 9.4 of \cite{zub3} and Proposition 3.4 of \cite{zubgrish}, for any $m\in M$ a supersubmodule generated by $m$ coincides with $Dist(H)m$. 
\begin{lm}\label{(co)universal}
Let $G_r B^-_w$-supermodule $M$ be cogenerated by a $B_{r, w}^+ T$-primitive element $u$ of weight $\lambda$ such that $M$ satisfies the condition
$M_{\mu}\neq 0$ implies $\mu\leq_w\lambda$.
Then it can be embedded into $\Hat{Z}'_{r, w}(\lambda^{|u|})$. 

Symmetrically, if $G_r B_w^+$-supermodule $N$ is generated by a $B^-_{r, w} T$-primitive element $v$ of weight $\lambda$, then it is an epimorphic image of
$\Hat{Z}_{r, w}(\lambda^{|v|})$. 
\end{lm}
\begin{proof}
The proof can be modified from Lemma 9.1 of \cite{zub2}.
\end{proof}
As a result of Lemma \ref{(co)universal}, we will say that $\Hat{Z}'_{r, w}(\lambda^a)$ is a {\it couniversal} object in $G_r B^-_w$-$\mathsf{SMod}$ 
and $\Hat{Z}_{r, w}(\lambda^a)$ is a {\it universal} object in $G_r B^+_w$-$\mathsf{SMod}$.
The same arguments as in Lemma \ref{(co)universal} show that $\mathcal{Z}'_{r, w}(\lambda^a)$ and $\mathcal{Z}_{r, w}(\lambda^a)$
are couniversal and universal objects in $G_r T$-$\mathsf{SMod}$, respectively and $Z'_{r, w}(\lambda^a)$ and $Z_{r, w}(\lambda^a)$ are couniversal and universal objects in $G_r$-$\mathsf{SMod}$.
\begin{lm}\label{simples}
Every irreducible $G_r B^-_w$-supermodule is isomorphic to a unique supermodule $\Hat{L}'_{r, w}(\lambda^a)$, and every irreducible $G_r B^+_w$-supermodule is isomorphic 
to a unique supermodule $\Hat{L}_{r, w}(\lambda^a)$. 
Additionally, every irreducible $G_r T$-supermodule is isomorphic to a unique supermodule $\mathcal{L}_{r, w}(\lambda^a)$, and every irreducible $G_r$-supermodule 
is isomorphic to a unique supermodule $L_{r, w}(\lambda^a)$. 
\end{lm}
\begin{proof}
Let $L$ be an irreducible $G_r B^-_w$-supermodule. Then $L$ is (co)generated by arbitrary element $v\in L^{U^+_{r, w}}\setminus 0$, 
and therefore $L=Dist(G_r B^-_w)v=Dist(B_w^-)v$. If $v$ has a weight $\lambda$, then $L_{\mu}\neq 0$ implies $\mu\leq_w\lambda$. 
Using Lemma \ref{(co)universal} we conclude that $L\simeq\Hat{L}'_{r, w}(\lambda^{|v|})$. The proof of the remaining cases is similar.
\end{proof}
According to \cite{zub2}, \S 7, the category $G$-$\mathsf{SMod}$ has a self-duality $M\to M^{<t>}$ induced by the supertransposition.  
\begin{lm}\label{equivalences}
The functor $M\to M^{<t>}$ induces an equivalence $G_r B_w^-$-$\mathsf{SMod}\simeq G_r B_w^+$-$\mathsf{SMod}$ and self-dualities of the categories $G_r T$-$\mathsf{SMod}$ and $G_r$-$\mathsf{SMod}$ such that $$\Hat{Z}'_{r, w}(\lambda^a)^{<t>}\simeq \Hat{Z}_{r, w}(\lambda^a), \ \mathcal{Z}'_{r, w}(\lambda^a)^{<t>}\simeq \mathcal{Z}_{r, w}(\lambda^a), \ Z'_{r, w}(\lambda^a)^{<t>}\simeq Z_{r, w}(\lambda^a).$$
As a consequence, there are natural isomorphisms
\[\Hat{L}'_{r, w}(\lambda^a)^{<t>}\simeq \Hat{L}_{r, w}(\lambda^a), \ \mathcal{L}_{r, w}(\lambda^a)^{<t>}\simeq \mathcal{L}_{r, w}(\lambda^a), \ L_{r, w}(\lambda^a)^{<t>}\simeq L_{r, w}(\lambda^a).\]
\end{lm}
\begin{proof}
All statements follow by (co)universality of the corresponding supermodules and by Lemma 5.4 of \cite{zub4}.
\end{proof}
Combining our Lemmas \ref{decompositions} and \ref{equivalences} with Theorem 10.1 of \cite{zub1} we obtain
\[\Hat{Z}'_{r, w}(\lambda^a)|_{G_r T}\simeq\mathcal{Z}'_{r, w}(\lambda^a), \ 
\Hat{Z}_{r, w}(\lambda^a)|_{G_r T}\simeq\mathcal{Z}_{r, w}(\lambda^a),\]
\[\Hat{Z}'_{r, w}(\lambda^a)|_{G_r}\simeq Z'_{r, w}(\lambda^a), \ 
\Hat{Z}_{r, w}(\lambda^a)|_{G_r}\simeq Z_{r, w}(\lambda^a).\]
Moreover, modifying the proof of Lemma 9.2 of \cite{zub2} one can show that
\[\mathsf{soc}_{G_r}\Hat{Z}'_{r, w} (\lambda^a)=\mathsf{soc}_{G_r T} \Hat{Z}'_{r, w} (\lambda)=\mathsf{soc}_{G_r B_w^-} \Hat{Z}'_{r, w} (\lambda^a)\] and
\[\mathsf{rad}_{G_r} \Hat{Z}_{r, w} (\lambda^a)=\mathsf{rad}_{G_r T} \Hat{Z}_{r, w} (\lambda^a)=\mathsf{rad}_{G_r B^+_w} \Hat{Z}_{r, w} (\lambda^a).\]

For ${\bf s}\in\{{\bf r}, {\bf l}\}$ denote by $\chi^-_{r, {\bf s}}(w)$ and $\chi^+_{r, {\bf s}}(w)$, respectively the characters corresponding to the actions 
of the supergroup $B_w^-$ on $\int_{{\bf s}, Dist(B^-_{r, w})}$ and of the supergroup $B_w^+$ on $\int_{{\bf s}, Dist(B^+_{r, w})}$, respectively. 

The following lemma generalizes Proposition II.3.4 of \cite{jan}.
\begin{lm}\label{someactions} The character $\chi^-_{r, {\bf r}}(w)$ equals $-2((p^r -1)\rho_0(w) +\rho_1(w))$ and the character 
$\chi^+_{r, {\bf r}}(w)$ equals $2((p^r -1)\rho_0(w) +\rho_1(w))$.
\end{lm}
\begin{proof}
For every $t\in T$ we have
\[Ber(Ad(t)|_{Lie(B_w^-)})=2(-\rho_0(w)+\rho_1(w))(t), \ \det((Ad(t)|_{Lie(B_w^-)})_{11})=-2\rho_1(w)(t)\]
and  
\[Ber(Ad(t)|_{Lie(B_w^+)})=2(\rho_0(w)-\rho_1(w))(t), \ \det((Ad(t)|_{Lie(B_w^+)})_{11})=2\rho_1(w)(t).\] 
The application of Proposition 5.11 of \cite{zubfr}, which is also valid for left integrals, concludes the proof. 
\end{proof}
\begin{rem}\label{parity} By Proposition 5.1 of \cite{zubfr} both superspaces $\int_{{\bf s}, Dist(B^-_{r, w})}$ and $\int_{{\bf s}, Dist(B^+_{r, w})}$ 
have the same parity $mn \pmod 2$. Denote this parity by $|mn|$.
\end{rem}
The following proposition superizes Lemma II.9.2 of \cite{jan}.
\begin{pr}\label{duality} For every $\lambda\in X(T)$ there are the isomorphisms
\[\begin{aligned}&\Hat{Z}_{r, w}(\lambda)\simeq ind_{B_w^+}^{G_r B_w^+} K^{|mn|}_{\lambda -2((p^r -1)\rho_0(w) +\rho_1(w))}\\
&\simeq \Pi^{|mn|}(\Hat{Z}_{r, w} (2((p^r -1)\rho_0(w) +\rho_1(w))-\lambda)^*\end{aligned}\] 
and 
\[\begin{aligned}&\Hat{Z}'_{r, w}(\lambda)\simeq coind^{G_r B_w^-}_{B_w^-} K^{|mn|}_{\lambda -2((p^r -1)\rho_0(w) +\rho_1(w))} \\
&\simeq \Pi^{|mn|}(\Hat{Z}'_{r, w} (2((p^r -1)\rho_0(w) +\rho_1(w))-\lambda)^*.\end{aligned}\]
Also, 
$\mathsf{ch}(\Hat{Z}_{r, w}(\lambda))=\mathsf{ch}(\Hat{Z}'_{r, w}(\lambda))=e^{\lambda}\prod_{\alpha\in\Phi_w^+ , p(\alpha)=0}\frac{1-e^{-p^r\alpha}}{1-e^{-\alpha}}
\prod_{\alpha\in\Phi_w^+, p(\alpha)=1}(1+e^{-\alpha}),$ 
$\Hat{Z}_{r, w}(\lambda+p^r\mu)\simeq \Hat{Z}_{r, w}(\lambda)\otimes p^r\mu \text{ and } \Hat{Z}'_{r, w}(\lambda+p^r\mu)\simeq \Hat{Z}'_{r, w}(\lambda)\otimes p^r\mu$
for all $\mu\in X(T)$.
\end{pr}
\begin{proof}
To prove the first isomorphism, combine our Proposition \ref{thesameforGH'} and our Remark \ref{parity} with Proposition 5.11 and Lemma 6.1 of \cite{zubfr}. 
Then Lemma \ref{equivalences} implies the second isomorphism. The statement about characters follows from our Lemma \ref{decompositions} and Lemma 8.2 and Remark 8.3 from \cite{zub1}. Finally, the last isomorphisms are trivial consequence of the tensor identity (cf. \cite{jan}, Lemma II.9.2(4, 5)).
\end{proof}
Let $w=w_0 w_1$ be the regular decomposition of $w$. Then \cite{zub1}, \S 7 implies $\rho_0(w)=\rho(w_0)=w_0\rho_0$ and $\rho_1(w)=w_0\rho_1(w_1)$.
\begin{lm}\label{characters}
For every $w\in S_{m+n}$, the formal characters of supermodules $\Hat{Z}_r(\lambda)$ and $\Hat{Z}_{r, w}(\lambda+(p^r -1)(\rho_0(w)-\rho_0)+(\rho_1(w)-\rho_1))$ coincide.
\end{lm}
\begin{proof}
For $a\in\mathbb{Z}_2$ denote by $(\Phi_w^+)_a$ the set $\{\alpha\in\Phi^+_w\mid p(\alpha)=a\}$.
Since 
\[\Phi_a=\Phi^+_a\sqcup -\Phi^+_a=(\Phi^+_w)_a\sqcup -(\Phi^+_w)_a\] for each $a\in\mathbb{Z}_2$, 
we can use Proposition \ref{duality} to modify the proof from \cite{jan}, II.9.3.
\end{proof}
\begin{rem}\label{othersupermodules}
Lemma \ref{characters} remains valid for the supermodules of types $\Hat{Z}', \mathcal{Z}, \mathcal{Z}', Z$ and $Z'$ as well.
\end{rem}
\begin{rem}\label{compositionfactors}
For every $\lambda\in X(T)$, the term $e^{\lambda}$ is the largest monomial of the formal character $\mathsf{ch}(\mathcal{L}_{r, w}(\lambda))$ 
with respect to the partial order $\leq_w$. In particular, all such characters are linearly independent.
Therefore if $\mathsf{ch}(M)=\sum a_{\lambda}\mathsf{ch}(\mathcal{L}_{r, w}(\lambda))$, then the non-negative integers $a_{\lambda}$ are uniquely defined by a $G_r T$-supermodule $M$and they coincide with multiplicities $[M : \mathcal{L}_{r, w}(\lambda)]$.  
\end{rem}
Denote by $\lambda<w>$ the weight $\lambda+(p^r -1)(\rho_0(w)-\rho_0)+(\rho_1(w)-\rho_1)$.
\begin{lm}\label{hom-s}
Let $(A,H)$ be one of the pairs $(\mathcal{Z}, G_r T), (\mathcal{Z}', G_r T), (Z, G_r)$ or $(Z', G_r)$. Then for every $w, w'\in S_{m+n}$ there is an isomorphism
\[Hom_H(A_{r, w}(\lambda<w>), A_{r, w'}(\lambda<w'>))\simeq K.\]
\end{lm}
\begin{proof}
Since $\mathsf{ch}(A_{r, w}(\lambda<w>))=\mathsf{ch}(A_{r, w'}(\lambda<w'>))$, one can modify the proof of Lemma 3 of \cite{sdoty} using Remark 5.3 of \cite{zub4}, 
applied to an arbitrary $B_w^{\pm}$.
\end{proof}
Let $\tilde{w}$ denote the longest element in $S_{m+n}$.
\begin{lm}\label{atopof}
There is an isomorphism $\mathcal{Z}_{r, \tilde{w}}'(\lambda<\tilde{w}>)\simeq\mathcal{Z}_r(\lambda^{|mn|})$. Consequently, 
$\mathcal{L}_r(\lambda)$ is isomorphic (up to a parity shift) to the top of the supermodule $\mathcal{Z}'_{r, \tilde{w}}(\lambda<\tilde{w}>)$.
\end{lm}
\begin{proof}
Since $B^-_{r, \tilde{w}}=B^+_r, \rho_0(\tilde{w})=-\rho_0$ and $\rho_1(\tilde{w})=-\rho_1$, it implies 
\[\mathcal{Z}'_{r, \tilde{w}}(\lambda<\tilde{w}>)=ind^{G_r T}_{B^+_r T} K_{\lambda-2((p^r -1)\rho_0+\rho_1)}.\]
We finish the proof by applying the first isomorphism from Proposition \ref{duality}, restricted on $G_r T$.
\end{proof}

\section{Minimal parabolic supersubgroups and adjacent Borel supersubgroups}

For every subset $S\subseteq\Pi_w$ one can define a parabolic supersubgroup $P_w(S)$ in the following way. 
If $S=\{\alpha_{i_1},\ldots , \alpha_{i_r}\}$, then $P_w(S)$ is equal to the stabilizer of the flag
\[W_{j_1}\subseteq W_{j_2}\subseteq\ldots \subseteq W_{j_{m+n-r-1}},\]
where $j_1 <\ldots < j_{m+n-r-1}$ is a listing of elements of the set $\{1, \ldots, m+n-1\}\setminus\{m+n-i_1, \ldots, m+n-i_r\}$.
For example, if $S=\{\alpha_i\}$, then $P_w(\alpha_i)$ is the stabalizer of the flag
\[W_1\subseteq\ldots\subseteq W_{m+n-i-1}\subseteq W_{m+n-i+1}\subseteq\ldots\subseteq W_{m+n}.\] 
Denote by $S_i$ the supersubspace $Kv_{wi}+Kv_{w(i+1)}$ and set $H_{i,w}=Stab_G(S_i)\cap (\cap_{j\neq i, i+1}Stab_G(Kv_{wj}))$. 
It is clear that $H_{i, w}\simeq GL(2)\times T'$ whenever $\alpha_i$ is even; and if $\alpha_i$ is odd then $H_{i, w}\simeq GL(1|1)\times T'$, where
$T'(A)=\{t\in T(A)| t|_{S_i\otimes 1}=id_{S_i\otimes 1}\}$ for every $A\in\mathsf{SAlg}_K$.

Let $UP_w(\alpha_i)$ be the largest supersubgroup of $U_w^-$ whose elements act trivially on $W_{m+n-i+1}/W_{m+n-i-1}$. 
It follows from Lemma 11.1 of \cite{zub1} that $P_w(\alpha_i)=UP_w(\alpha_i)\rtimes H_{i, w}$. 
Moreover, for every positive integer $r$, the supersubgroup $P_{r, w}(\alpha_i) T$ can be decomposed as $UP_{r, w}(\alpha_i)\rtimes (H_{i, w})_r T$. 
\begin{rem}\label{coincidenceforodd}
It $\alpha_i$ is odd, then $(H_{i, w})_r T=H_{i, w}T$.
\end{rem}
Two Borel supergroups $B^-_w$ and $B^-_{w'}$ are {\it adjacent} via $\alpha_i\in\Pi_w$ if  $\Phi^+_{w'}=\Phi_w^+\setminus\{\alpha_i\}\bigcup\{-\alpha_i\}$. If $\alpha_i$ is odd, then 
$B^-_w$ and $B^-_{w'}$ are called {\it odd adjacent}; if $\alpha_i$ is even, they are called {\it even adjacent}.
It is clear that $UP_w(\alpha_i)\leq B^-_w$ and $B^-_{w'}\leq P_w(\alpha)$. Moreover, 
\[B^-_w=UP_w(\alpha_i)\rtimes (\mathcal{B}^-\times T') \text{ and }  B^-_{w'}=UP_w(\alpha_i)\rtimes (\mathcal{B}^+\times T'),\]
where $\mathcal{B}^-$ and $\mathcal{B}^+$ are the corresponding Borel supersubgroups of $GL(1|1)$ or of $GL(2)$, depending on the parity of $\alpha_i$ (cf. \cite{zub1}, \S 11). 
Analogously, we have
\[B^-_{r, w}=UP_{r, w}(\alpha_i)\rtimes (B_r^-\times T'_r) , \ B^-_{r, w'}=UP_{r, w}(\alpha_i)\rtimes (B^+_r\times T'_r),\]
\[B^-_{r, w} T=UP_{r, w}(\alpha_i)\rtimes (B_r^- T) \text{ and } B^-_{r, w'}=UP_{r, w}(\alpha_i)\rtimes (B_r^+ T).\]

To simplify our notations we will omit the subindices $w, w'$ and $i$. For example, $B_w^-$ and $B_{w'}^-$ are denoted just by $B$ and $B'$ correspondingly, 
$P_w(\alpha_i)$ by $P(\alpha)$ etc. 
Let $L^{(r)}_{P(\alpha)}(\lambda^a)$ denote the socle of $ind^{P(\alpha)_r T}_{B_r T} K^a_{\lambda}$ for $\lambda\in X(T)$ and $a\in\mathbb{Z}_2$. 
Since $P(\alpha)_r T\simeq \mathcal{U}^+_r\times B_r T$, where $\mathcal{U}^+_r$ is the unipotent radical of $\mathcal{B}^+_r$, the arguments analogous to those in the proof of Lemma \ref{soclesandtops} show that $L^{(r)}_{P(\alpha)}(\lambda^a)$ is irreducible.
\begin{pr}\label{inducedoverparabolic}
The following statements are valid.
\begin{enumerate}
\item $UP(\alpha)_r$ acts trivially on $ind^{P(\alpha)_r T}_{B_r T} K_{\lambda}^a$.
\item If $\alpha$ is odd, then $ind^{P(\alpha)_r T}_{B_r T} K^a_{\lambda}=L^{(r)}_{P(\alpha)}(\lambda^a)$ if and only if $p\not|(\lambda, \alpha)$. Moreover, in this case there is a unique (up to a scalar multiple) isomorphism 
\[g_{\alpha} : ind^{P(\alpha)_r T}_{B'_r T} K_{\lambda-\alpha}^{a+1}\to ind^{P(\alpha)_r T}_{B_r T} K^a_{\lambda}.\]
\item If $\alpha$ is odd and $p|(\lambda, \alpha)$, then $ind^{P(\alpha)_r T}_{B_r T} K^a_{\lambda}$ has a composition series
\[\begin{array}{c}
L^{(r)}_{P(\alpha)}((\lambda-\alpha)^{a+1}) \\
| \\
L^{(r)}_{P(\alpha)}(\lambda^a)
\end{array}
\]
and $ind^{P(\alpha)_r T}_{B'_r T} K^a_{\lambda}$ has a composition series
\[\begin{array}{c}
L^{(r)}_{P(\alpha)}((\lambda+\alpha)^{a+1}) \\
| \\
L^{(r)}_{P(\alpha)}(\lambda^a)
\end{array}.
\]
In addition, there is a unique (up to a scalar multiple) non-zero homomorphism 
\[g_{\alpha} : ind^{P(\alpha)_r T}_{B'_r T} K_{\lambda-\alpha}^{a+1}\to ind^{P(\alpha)_r T}_{B_r T} K^a_{\lambda}\]
such that its kernel and cokernel are both isomorphic to $L^{(r)}_{P(\alpha)}((\lambda-\alpha)^{a+1})$.
\item If $\alpha$ is even, then there is an unique (up to a scalar multiple) homomorphism
\[g_{\alpha} : ind^{P(\alpha)_r T}_{B'_r T} K_{\lambda-(p^r -1)\alpha}^a\to ind^{P(\alpha)_r T}_{B_r T} K^a_{\lambda}\]
such that the image of $g_{\alpha}$ is $L^{(r)}_{P(\alpha)}(\lambda^a)$.
\item If $\alpha$ is even and $L^{(r)}_{P(\alpha)}(\mu^b)$ is a composition factor of $ind^{P(\alpha)_r T}_{B_r T} K^a_{\lambda}$, where $b$ denotes the parity of the highest weight vector, then $\mu$ is $(\alpha, r)$-linked to $\lambda$ in the sense of \cite{sdoty}.  
Here the linkage is with respect to $B_{ev}=(B_{w_0})_{ev}$ and the corresponding $\rho_0(w)=\rho_0(w_0)$ is the half sum of positive roots of $B_{ev}$.
\end{enumerate}
\end{pr}
\begin{proof}
To prove the first three statements one can use our Remark \ref{coincidenceforodd} to modify the proofs of Proposition 11.5 and Proposition 12.2(1) of \cite{zub1}. To prove the last two statements we use Lemma 10.4 and Corollary 10.2 of \cite{zub1} to show that $ind^{P(\alpha)_r T}_{B_r T} K^a_{\lambda}\simeq ind^{H_r\mathcal{T}}_{\mathcal{B}_r\mathcal{T}} K_{\lambda}^a|_{\mathcal{B}}$, where $\mathcal{T}=T\cap\mathcal{B}$. Then we refer to Lemma 1 and to the proof of Proposition 1 (see the penultimate paragraph on page 138) of \cite{sdoty}.
\end{proof}

\section{Strong linkage principle for $G_r T$}

As an intermediate step towards proving the linkage principle for $G$, we establish a corresponding result for the Frobenius thickening $G_r T$.

Let us start with the observation that Remark \ref{compositionfactors} combined with Lemma \ref{characters} implies that, for every $w\in S_{m+n}$, $\mathcal{L}_r(\mu)$ is a composition factor of $\mathcal{Z}'_r(\lambda)$ if and only if it is a composition factor of $\mathcal{Z}'_{r, w}(\lambda<w>)$. 

Fix an ordering $\alpha_1, \ldots, \alpha_{mn}$ of all odd roots from $\Phi^+$ such that $\alpha_i< \alpha_j$ implies $i< j$. According to \cite{brunkuj}, \S 4, the longest element $\tilde{w}_1$ of $D_{m, n}$ can be written as $s_{\alpha_{mn}}s_{\alpha_{mn-1}}\ldots s_{\alpha_1}$. 
Analogously, we fix an ordering of the even roots $\beta_1, \ldots, \beta_N$, where $N=\frac{m(m-1)}{2}+\frac{n(n-1)}{2}$, from $\Phi^+$. Then $\tilde{w}_0$ can be written as 
$s_{\beta_N}\ldots s_{\beta_1}$.

Set $y_0=1$, $y_i=s_{\alpha_i}\ldots s_{\alpha_1}$ for $1\leq i\leq mn$, and $ y_j=s_{\beta_{j-mn}}\ldots s_{\beta_1}y_{mn}$ for $mn < j\leq mn +N$. Then, for every $1\leq i\leq m$, 
$B^-_{y_i}$ is odd adjacent to $B_{y_{i-1}}^-$ via $\alpha_i$; and for every $mn < j\leq mn +N$, $B^-_{y_j}$ is even adjacent to $B^-_{y_{j-1}}$ via $\beta_{j-mn}$.
\begin{lm}\label{f_i} For $1\leq i\leq mn+N$ there is a unique (up to a scalar multiple) homomorphism 
$f_i : \mathcal{Z}'_{r, y_i}(\lambda<y_i>)\to\mathcal{Z}'_{r, y_{i-1}}(\lambda<y_{i-1}>)$ constructed as follows. 
Let $P$ denote $P_{r, y_{i-1}}(\alpha_i)$ or $P_{r, y_{i-1}}(\beta_{i-mn})$, respectively provided $i\leq mn$ or $mn<i$, respectively. Then $f_i$ is given as follows.
\begin{enumerate}
\item If $1\leq i\leq mn$ and $p\not| (\lambda+\rho, \alpha_i)$, then $f_i=ind^{G_r T}_{P_r T} g_{\alpha_i}$ is an (even or odd) isomorphism.
\item If $1\leq i\leq mn$ and $p|(\lambda+\rho, \alpha_i)$, then $f_i=ind^{G_r T}_{P_r T} g_{\alpha_i}$ is an (even or odd) homomorphism such that its kernel and cokernel are isomorphic (up to a parity shift) to $ind^{G_r T}_{P_r T} L^{(r)}_P((\lambda-\alpha_i)<y_{i-1}>)$.
\item If $i > mn$, then $f_i=ind^{G_r T}_{P_r T} g_{\beta_{i-mn}}$.
\end{enumerate}
\end{lm}
\begin{proof}
First observe that $\lambda<y_i>=\lambda<y_{i-1}>-\alpha_i=(\lambda-\alpha_i)<y_{i-1}>$ for $1\leq i\leq mn$ and $\lambda<y_i>=\lambda<y_{i-1}>-(p^r -1)\beta_{i-mn}$
for $mn+1\leq i\leq mn+N$.

Assume $i\leq mn$. Since $\alpha_i$ belongs to $\Pi_{y_{i-1}}$, Proposition 1.28 of \cite{chengwang} implies
\[(\lambda<y_{i-1}>, \alpha_i)\equiv (\lambda+\rho-\rho(y_{i-1}), \alpha_i)=(\lambda+\rho, \alpha_i)\pmod p.\]
Therefore $p|(\lambda<y_{i-1}>, \alpha_i)$ if and only if $p|(\lambda+\rho, \alpha_i)$. 

Combine Theorem 10.1 of \cite{zub1} and Theorem 0.1 of \cite{zub5} to obtain that $G_r T/P_r T\simeq G_r/P_r$ are affine superschemes. 
Therefore, the functor $ind^{G_r T}_{P_r T}$ is faithfully exact by Theorem 5.2 of \cite{zub3}, and all statements follow by Proposition \ref{inducedoverparabolic}. 
\end{proof}

The next proposition inspired by the work of Doty describes a single step of the strong linkage for $G_r T$. 

\begin{pr}\label{auxiliaryprop}
Assume that $\mathcal{L}_r(\mu)$ (or $\Pi\mathcal{L}_r(\mu)$) is a composition factor of $\mathcal{Z}'_r(\lambda)$. If $\mu\neq\lambda$, then either there is 
$\lambda' <\lambda$ such that $L_r(\mu)$ is a composition factor of $\mathcal{Z}_r'(\lambda')$ and $\lambda'$ is $(\alpha, r)$-linked to $\lambda$, 
where $\alpha\in\Phi^+$ and $p(\alpha)=0$, or $\mathcal{L}_r(\mu)$ is a composition factor of
$\mathcal{Z}'_r(\lambda-\alpha)$, where $\alpha\in\Phi^+$, $p(\alpha)=1$ and $p|(\lambda+\rho, \alpha)$. 
\end{pr}
\begin{proof}
By Lemma \ref{atopof}, $\mathcal{L}_r(\mu)$ is a composition factor of $\mathcal{Z}'_{r, \tilde{w}}(\lambda<\tilde{w}>)$. Set $f=f_{mn+N}\cdot f_{mn+N-1}\cdot\ldots\cdot f_1$. If $f\neq 0$, then by Lemma \ref{atopof} the image of $f$ equals $\mathcal{L}_r(\lambda)$, the socle of $\mathcal{Z}_r(\lambda)$. Since $\mu\neq\lambda$, there is a positive integer $i$ such that $\mathcal{L}_r(\mu)$ is a composition factor of $\ker f_i$. 

If $i\leq mn$, then by Lemma \ref{f_i}(2), $p$ divides $(\lambda+\rho, \alpha_i)$ and $\mathcal{L}_r(\mu)$ is a composition factor of 
\[ind^{G_r T}_{P_r T} L^{(r)}_P((\lambda-\alpha_i)<y_{i-1}>)\subseteq
\mathcal{Z}'_{r, y_{i-1}}((\lambda-\alpha_i)<y_{i-1}>),\]
where $P=P_{y_{i-1}}(\alpha_i)$. 
Therefore Remarks \ref{othersupermodules} and \ref{compositionfactors} imply that $\mathcal{L}_r(\mu)$ is a composition factor of $\mathcal{Z}'_r(\lambda-\alpha_i)$. 

If $i> mn$, then we can use Proposition \ref{inducedoverparabolic}(4, 5) and Lemma \ref{f_i}(3) to modify the proof of Proposition 1 from \cite{sdoty}.
\end{proof}

The strong linkage principle for $G_r T$ now follows easily.

\begin{tr}\label{linkageoverG_r T}
If $\mathcal{L}_r(\mu)$ (or $\Pi\mathcal{L}_r(\mu)$) is a composition factor of $\mathcal{Z}'_r(\lambda)$ and $\mu\neq\lambda$, then there is a sequence $\lambda=\lambda_0, \lambda_1, \ldots, \lambda_t=\mu$ such that for each $1\leq i\leq t$ either $\lambda_{i}=\lambda_{i-1}-\alpha_i$, where $\alpha_i\in\Phi^+, p(\alpha_i)=1$ and 
$p|(\lambda_{i-1}+\rho, \alpha_i)$, or $\lambda_i <\lambda_{i-1}$ and $\lambda_i$ is $(\alpha_i, r)$-linked with $\lambda_{i-1}$, where $\alpha_i\in\Phi^+$ and $p(\alpha_i)=0$. 
\end{tr}
\begin{proof}
By Proposition \ref{auxiliaryprop}, there is $\lambda_1$ such that $\lambda_1 <\lambda$ and $\mathcal{L}_r(\mu)$ is a composition factor of $\mathcal{Z}'_r(\lambda_1)$. Then 
either $\lambda$ is $(\alpha, r)$-linked with $\lambda_1$ for an even positive root $\alpha$, or
$\lambda_1=\lambda-\alpha$ for an odd positive root $\alpha$ such that $p|(\lambda+\rho, \alpha)$.
Repeating the same arguments for $\lambda_1$ we will find $\lambda_2$ etc. Since $\mu\leq\ldots <\lambda_2 <\lambda_1 <\lambda$ and the interval $\{\pi\mid \mu\leq\pi\leq\lambda\}$ is finite, the theorem follows.
\end{proof}
\begin{rem}\label{replacing}
The statement of Theorem \ref{linkageoverG_r T} remains true if we replace $\mathcal{Z}'_r(\lambda)$ by $\mathcal{Z}_r(\lambda)$.
\end{rem}

\section{Blocks over $GL(m|n)$}

Recall from \cite{zub4} that every irreducible $G$-supermodule $L$ is uniquely defined by its highest weight $\lambda$ and by the parity $a$ of a primitive vector $v$ 
of weight $\lambda$; that is $L=L(\lambda^a)$, where $a\in\mathbb{Z}_2$. 
We have already observed that
\[Ext^1_G(L(\lambda^a), L(\mu^b))\simeq \Pi^{a+b}Ext^1_G(L(\lambda), L(\mu)),\] 
where $\Pi$ is a parity shift functor, which implies that blocks of simple $G$-supermodules correspond uniquely to equivalence classes of dominant weights $X(T)^+$, 
that we will also call blocks.
Therefore, instead of working with blocks of simple supermodules one can work with blocks of dominant weights. The block containing a dominant weight $\lambda$ will be denoted by 
$B(\lambda)$.

We start with a simple technical result.
\begin{lm}\label{Verma}
If $r\leq r'$, then the $G_r T$-supermodule $\mathcal{Z}_r(\lambda)$ is naturally embedded into $\mathcal{Z}_{r'}(\lambda)|_{G_r T}$. 
Moreover, a superspace $\lim\limits_{\rightarrow}\mathcal{Z}_r(\lambda)$ has a natural structure of the $Dist(G)$-supermodule such that
\[\lim\limits_{\rightarrow}\mathcal{Z}_r(\lambda)\simeq M(\lambda)=Dist(G)\otimes_{Dist(B^+)} K_{\lambda}.\]
\end{lm}
\begin{proof}
Apply the arguments from the proof of Lemma \ref{decompositions}.
\end{proof}
Let $X_r(T)^+$ denote the set \[\{\lambda\in X(T)\mid 0\leq \lambda_i-\lambda_{i+1}\leq p^r-1 \text{ for } 1\leq i\leq m+n, i\neq m\}.\]
It is clear that $X_r(T)^+\subseteq X(T)^+$ and for every dominant weight $\lambda\in X(T)^+$ there is a positive integer $r$ such that $\lambda\in X_{r'}(T)^+$ for all $r'\geq r$. Moreover, \cite{zubfr}, \S 7 and \cite{shuweiq}, \S 4 imply that $L(\lambda)|_{G_r}=L_r(\lambda)$ and $L(\lambda)|_{G_r T}=\mathcal{L}_r(\lambda)$. 

The following proposition captures the essence of the linkage for $G$.
\begin{pr}\label{stronglinkageforGL}
If $L(\mu)$ (or $\Pi L(\mu)$) is a composition factor of the Weyl supermodule $V(\lambda)$, then 
there is a positive integer $r$ and a sequence $\lambda=\lambda_0, \lambda_1, \ldots, \lambda_t=\mu$ such that for each $1\leq i\leq t$
either $\lambda_{i}=\lambda_{i-1}-\alpha_i$, where $\alpha_i\in\Phi^+, p(\alpha_i)=1$ and $p|(\lambda_{i-1}+\rho, \alpha_i)$, or $\lambda_i <\lambda_{i-1}$ and $\lambda_i$ is $(\alpha_i, r)$-linked with $\lambda_{i-1}$, where $\alpha_i\in\Phi^+$ and $p(\alpha_i)=0$.
\end{pr}
\begin{proof} By Lemma 4.1 (iii) of \cite{brunkuj}, $V(\lambda)=M(\lambda)/N$ for an appropriate supermodule $N$. 
Using Lemma \ref{Verma}, there is a sufficiently large positive integer $r$ such that $\lambda, \mu\in X_r(T)^+$ and 
$V(\lambda)\simeq\mathcal{Z}_r(\lambda)/(\mathcal{Z}_r(\lambda)\cap N)$. The claim then follows using Remark \ref{replacing}.
\end{proof}
\begin{rem}\label{notdominant}
Weights $\lambda_i$ from Proposition \ref{stronglinkageforGL} are not necessarily dominant.
\end{rem}

The blocks in the category of $G_{res}=GL(m)\times GL(n)$-modules can be described as follows.

Let $\lambda$ be a (not necessarily dominant) weight. Define the {\it defect} of $\lambda$ as a pair
$d(\lambda)=(d_+(\lambda)\mid d_-(\lambda))$, where
\[d_+(\lambda)=\max\{d\geq 0\mid (\lambda+\rho_0, (\epsilon_i-\epsilon_j)^{\vee})\equiv 0\!\!\!\!\pmod{p^d} \text{ for } 1\leq i< j\leq m\},\] 
\[d_-(\lambda)=\max\{d\geq 0\mid (\lambda+\rho_0, (\epsilon_i-\epsilon_j)^{\vee})\equiv 0\!\!\!\!\pmod{p^d} \text{ for }  m+1\leq i< j\leq m+n\}.\]

Denote by $B_{ev}(\lambda)$ that block in the category of $G_{res}=GL(m)\times GL(n)$-modules which contains $\lambda\in X(T)^+$. Denote
\[p^{d(\lambda)+(1\mid 1)}\mathbb{Z}\Phi=\sum_{1\leq i\neq j\leq m}p^{d_+(\lambda)+1}\mathbb{Z}(\epsilon_i-\epsilon_j) +\sum_{m+1\leq i\neq j\leq m+n}p^{d_-(\lambda)+1}\mathbb{Z}(\epsilon_i-\epsilon_j)\]
and the dot action $W=S_m\times S_n$ on $X(T)$ as $w.\mu=w(\mu+\rho_0)-\rho_0$. 
It follows from \cite{jan}, II.7.2(3) that $B_{ev}(\lambda)=(W.\lambda+p^{d(\lambda)+(1\mid 1)}\mathbb{Z}\Phi)\cap X(T)^+$.

We call dominant weights $\lambda$ and $\mu$ {\it even-linked} if $B_{ev}(\lambda)=B_{ev}(\mu)$. From the above description it follows that the defects of even-linked weights coincide.

Next, we will derive a few simple results to rectify the deficiency in Proposition \ref{stronglinkageforGL} and show that we can replace weights $\lambda_i$ by dominant weights.

Let $\lambda$ be a (not necessarily dominant) weight. A {\it companion} of $\lambda$ is any dominant weight $\mu$ such that $d(\mu)=d(\lambda)$ and
\[W.\lambda+p^{d(\lambda)+(1\mid 1)}\mathbb{Z}\Phi =W.\mu+p^{d(\lambda)+(1\mid 1)}\mathbb{Z}\Phi .\]
If $\lambda$ is dominant, then $\lambda$ and $\mu$ are obviously even-linked.
\begin{lm}\label{companion}
Any weight $\lambda$ has a companion.
\end{lm}
\begin{proof}
Set $\omega_i=\sum_{1\leq k\leq i}\epsilon_k$ for $1\leq i\leq m$ and $\omega_j=\sum_{m+1\leq k\leq j}\epsilon_k$ for $m+1\leq j\leq m+n$. Then 
\[\lambda=\sum_{1\leq i <m}(\lambda_i-\lambda_{i+1})\omega_i +\lambda_m\omega_m +
\sum_{m+1\leq j < m+n}(\lambda_j-\lambda_{j+1})\omega_j +\lambda_{m+n}\omega_{m+n}.\]
If $\lambda$ is not dominant, then there is $i< m$ or $j < m+n$ such that $\lambda_i-\lambda_{i+1}< 0$ or $\lambda_j-\lambda_{j+1}< 0$. 
Let $A$ be the maximal number of all absolute values of such negative differences. Choose a positive integer $t$ such that $t> d_{\pm}(\lambda)$ and $p^t > A$ and denote  
$\sum_{1\leq i< m}\omega_i$ by $\pi_+$ and $\sum_{m+1\leq j< m+n}\omega_j$ by $\pi_-$. 

Then the required companion of $\lambda$ is
\[\mu=\lambda +p^t\pi_+ +p^t\pi_- -p^t|\pi_+|\epsilon_m-p^t|\pi_-|\epsilon_{m+n}.\]
\end{proof}
\begin{cor}\label{allcompanionsareevenlinked}
Any companion of $\mu$ from $W.\lambda+p^{d(\lambda)+(1|1)}\mathbb{Z}\Phi$ is even-linked with any companion of $\lambda$. In particular, $d(\mu)=d(\lambda)$.
\end{cor}
\begin{proof}
Construct a companion $\mu'$ of $\mu$ as in Lemma \ref{companion} such that $t$ is bigger than $d_{\pm}(\lambda)$. 
Then $\mu'\in W.\lambda+p^{d(\lambda)+(1|1)}\mathbb{Z}\Phi=W.\lambda'+p^{d(\lambda)+(1|1)}\mathbb{Z}\Phi$, where $\lambda'$ is a companion of $\lambda$.
\end{proof}
Following \cite{sdoty}, for a given weight $\lambda$ and a positive even root $\alpha$ we define a {\it lower} $p^e$-reflection $R_{\alpha, e}$ as 
$s_{\alpha, ap^e}$, where $(\lambda+\rho_0, \alpha^{\vee})=ap^e+s$ and $0\leq s < p^e$. Since $s_{\alpha, m}$ acts on $X(T)$ by the rule
$s_{\alpha, m}(\mu)=s_{\alpha}.\mu+m\alpha, m\in\mathbb{Z}$, we obtain $R_{\alpha, e}(\lambda)=\lambda-s\alpha$. 
\begin{lm}\label{aboutDonkin'scriterion}
A companion of $R_{\alpha, e}(\lambda)$ is even-linked with a companion of $\lambda$.
\end{lm}
\begin{proof}
Assume first that $\alpha=\epsilon_i-\epsilon_j$, where $1\leq i<j\leq m$. 

If $\lambda=R_{\alpha, e}(\lambda)$, then the statement is trivial. Otherwise, $s>0$ and
since $(\lambda+\rho_0, \alpha^{\vee})\equiv 0 \!\pmod {p^{d_+(\lambda)}}$, we also have $e >d_+(\lambda)$. 
Thus $R_{\alpha, e}(\lambda)\in W.\lambda+p^{d(\lambda)+(1|1)}\mathbb{Z}\Phi$ and Corollary \ref{allcompanionsareevenlinked} concludes the proof.

The case $\alpha=\epsilon_i-\epsilon_j$, where $m+1\leq i<j\leq m+n$ is analogous.
\end{proof}
The weights $\lambda$ and $\mu$ are called {\it simply-odd-linked} if there is a positive odd root $\alpha$ such that $\mu=\lambda\pm\alpha$ and $p|(\lambda+\rho, \alpha)$.
If $\alpha=\epsilon_i-\epsilon_j$, then \[(\lambda+\rho,\alpha) =\lambda_i+\lambda_j+2m+1-i-j\] and $\lambda$ and $\lambda-\alpha$ are simply-odd-linked provided 
\[\lambda_i+\lambda_j+2m+1-i-j \equiv 0 \pmod p.\]

\begin{lm}\label{oddlinked}
Assume that $\lambda$ and $\mu$ are simply-odd-linked. Then there is a companion $\lambda'$ of $\lambda$ such that $p|(\lambda'+\rho, \alpha)$ and $\lambda'-\alpha$ is a companion of $\mu$.
\end{lm}
\begin{proof}
Using the same trick as in Lemma \ref{companion} one can find an element $\pi\in p^{t|t}\mathbb{Z}\Phi$ such that $\lambda'=\lambda+\pi$ is a companion of $\lambda$ and $\mu+\pi=\lambda'-\alpha$ is a companion of $\mu$. 
\end{proof}

Now we are ready to prove the linkage principle for $GL(m|n)$.

\begin{pr}\label{achain}
If dominant weights $\lambda$ and $\mu$ are linked, then there is a chain of dominant weights 
$\lambda=\lambda_0, \lambda_1, \ldots, \lambda_s=\mu$ such that for each $1\leq i\leq s-1$ either $\lambda_i$ and $\lambda_{i+1}$ are even-linked or 
$\lambda_i$ and $\lambda_{i+1}$ are simply-odd-linked. 
\end{pr}
\begin{proof}
By our Proposition \ref{stronglinkageforGL} and Definition 1 from \cite{sdoty}, there is a chain
$\lambda=\lambda_0, \lambda_1, \ldots, \lambda_s=\mu$ such that either $\lambda_i$ and $\lambda_{i+1}$ are simply-odd-linked or one of these two weights is obtained from the other one by the action of the reflection $R_{\alpha, e}$. Using Lemma \ref{aboutDonkin'scriterion} and Lemma \ref{oddlinked} we can replace all members of the above chain by their companions and the claim follows.    
\end{proof}
Let $\approx$ denote an equivalence on $X(T)^+$ such that $\lambda\approx\mu$ if and only if $\lambda$ and $\mu$ can be connected by a chain from Proposition \ref{achain}. 
Proposition \ref{achain} implies that each block $B(\lambda)$ is contained in the equivalence class $\approx(\lambda)$ of $\lambda$. 

Lemma 5.3 of \cite{marko} states that if dominant weights $\lambda$ and $\mu$ are even-linked, then they belong to the same block of $G$. 
Therefore to prove that the equivalence class of $\approx(\lambda)$ is contained in the block $B(\lambda)$, we only need to show that if $\lambda$ and $\mu$ are odd-linked, then they belong to the same block. The rest of the paper is devoted to proving this statement.

If the weights $\lambda$ and $\mu$ are simply-odd-linked, the 
either $\lambda=\mu-\alpha$ belongs to $H^0(\mu)$ or $\mu=\lambda-\alpha$ belongs to $H^0(\lambda)$. We will assume the latter since the former case is analogous.

The $G_{ev}$-structure of the induced supermodule $H^0(\lambda)$ has been studied in \cite{marko}. The main tool there was the identification of $H^0(\lambda)$ with 
$H^0_{ev}\otimes \Lambda(Y)$ for a suitable supermodule $Y$ (see \S 1.2 of \cite{marko}). The $G_{ev}$-module $H^0(\lambda)$ is a direct sum of floors 
$F_k=H^0_{ev}\otimes \Lambda^k(Y)$ for $k=0, \ldots, mn$. There is a $G_{ev}$-morphism $\phi_1:F_1 \to F_1$ (see  \S 3.5 of \cite{marko}) such that its image is spanned by
$(w)_{ij}D$ for $1\leq i\leq m<j\leq m+n$, where $w\in H^0_{ev}(\lambda)$ and $_{ij}D$ are odd superderivation corresponding to elements of $Dist(V^+)$. 
\begin{lm}\label{ql1}
The image $\phi_1(F_1)$ of the map $\phi_1$ is $\langle F_0 \rangle_G \cap F_1$.
\end{lm}
\begin{proof}
Since the image $\phi_1(F_1)$ is the span of all elements $(w)_{ij}D$, where $w\in F_0$ and $1\leq i\leq m<j\leq m+n$, it is clear that
$\phi_1(F_1) \subset \langle F_0 \rangle_G \cap F_1$.

For the opposite inclusion, use the Poincare-Birkhoff-Witt theorem and order the elements of the distribution algebra $Dist(G)$ by listing elements from $Dist(V^+)$ first, followed by elements of $Dist(G_{ev})$ and finally by elements of $Dist(V^-)$. Every superderivation from $Dist(V^+)$ annihilates each element from $F_0$ and every superderivation from 
$Dist(G_{ev})$ sends elements from $F_0$ to itself. If we apply one superderivation from $Dist(V^-)$ to an element of $F_0$, the image lies in $\phi_1(F_1)$. 
If we apply more than one superderivation from $Dist(V^-)$ to an element of $F_0$, the image lies in the higher floors $F_k$ for $k>1$.
Therefore the opposite inclusion $\langle F_0 \rangle_G \cap F_1\subset \phi_1(F_1)$ is also valid.
\end{proof}

\begin{lm}\label{ql3}
Assume that the simple module $L_{ev}(\mu)$ is a composition factor of the module $F_1/Im \, \phi_1$. Then the simple supermodule $L(\mu)$ is a composition factor 
of $\langle F_0,F_1\rangle_G \subset H^0(\lambda)$.
\end{lm}
\begin{proof}
If $L_{ev}(\mu)$ is a composition factor of the module $F_1/Im \, \phi_1$, then there is a filtration 
\[Q_0=Im \, \phi_1 \subset Q_1 \subsetneq Q_2 \subset Q_3=F_1\]
such that $Q_2/Q_1\cong L_{ev}(\mu)$.

If we denote by $R_i$ the supermodule $\langle F_0, Q_i \rangle_G$ for $i=0, \ldots 3$, then using Lemma \ref{ql1} we obtain the corresponding filtration of supermodules
\[R_0=\langle F_0 \rangle_G \subset R_1 \subset R_2 \subset R_3=\langle F_0,F_1\rangle_G.\]

Use the Poincare-Birkhoff-Witt theorem and order elements of $Dist(G)$ in such a way that we elements from $Dist(G_{ev})$ come first, followed by 
elements from $Dist(V^+)$ and finally by elements from $Dist(V^-)$. 
We will show that $R_i\cap F_1 = Q_i$ for each . If we apply any superderivation from $Dist(G_{ev})$ to an element from $Q_i$, the image stays in $Q_i$ since $Q_i$ is an $G_{ev}$-module.
Application of a superderivation from $Dist(V^+)$ to an element of $Q_i\subset F_1$ gives an element in $F_0$. Finally, it follows from the definition of the map
$\phi_1$ that when we apply a superderivation from $Dist(V^-)$ to an element from $F_0$ we end up in $Im \phi_1=Q_0 \subsetneq Q_i$. The image of a compositions of superderivations from $Dist(V^-)$ to an element from $Q_i$ will lie in higher floors of $F_k$ for $k>1$. Therefore $R_i\cap F_1 = Q_i$. 

Consequently, \[R_0=\langle F_0 \rangle_G \subset R_1 \subsetneq R_2 \subset \langle F_0,F_1\rangle_G.\]
Since the superfactormodule $R_2/R_1$ is generated by the highest vector of weight $\mu$ and $(R_2/R_1) \cap F_1 \cong L_{ev}(\mu)$, 
we infer that $R_2/R_1$ contains the supermodule $L(\mu)$ as a composition factor and the claim follows.
\end{proof}

If $\alpha=\epsilon_i-\epsilon_j$ is an odd root, then $\lambda-\alpha$ is dominant implies that $\lambda_i >\lambda_{i+1}$ provided $1\leq i <m$ and 
$\lambda_i>0$ if $i=m$, and $\lambda_j>\lambda_{j+1}$ provided $m+1\leq j <m+n$.

Lemma 3.1 of \cite{marko} implies that the $G_{ev}$-module $F_1$ has a good filtration, in which factors are the modules 
$H^0_{ev}(\lambda-\beta)$, where $\beta$ is a positive odd root, and each dominant weight $\lambda-\beta$ appears with single multiplicity. 
Therefore all $G_{ev}$-primitive vectors in $F_1$ are scalar multiples of primitive elements $v_{\beta}$ of weight $\lambda-\beta$, where  
$v_{\beta}$ for $\beta=\epsilon_k-\epsilon_l$ was denoted by $\pi_{kl}$ in \cite{marko}. 

The last piece of the puzzle is revealed in the following theorem.

\begin{tr}\label{link1} Assume weights $\lambda, \mu$ are dominant, $\mu=\lambda-\alpha$, where $\alpha$ is a positive odd root and $\lambda$ is simply-odd-linked to $\mu$. 
Then $L(\mu)$ is a composition factor of $H^0(\lambda)$.
\end{tr}
\begin{proof}
Since $\mu$ is dominant, according to \cite{marko}, there is a $G_{ev}$-primitive element $v_{\alpha}\in H^0(\lambda)$ of weight $\mu$ which belongs to $F_1$. 

Denote the kernel of $\phi_1$ by $K_1$, its image by $I_1$ and its cokernel by $C_1$.
Denote by $[M]$ the element of the Grotendieck group corresponding to the $G_{ev}$-module $M$.
Since $[F_1]=[K_1]+[I_1]$, and $[F_1]=[I_1]+[C_1]$, we obtain that $[K_1]=[C_1]$; 
hence the $G_{ev}$-composition factors (and their multiplicities) of $C_1$ are identical to those of $K_1$.

By Lemma 3.6 of \cite{marko}, $\phi_1(v_{\alpha})=(\lambda+\rho,\alpha)v_{\alpha}=0$. Since $v_{\alpha}$ is a nonzero element of $K_1$, 
the module $L_{ev}(\mu)$ is a composition factor of $K_1$, and hence of $C_1$.
By Lemma \ref{ql3}, $L(\mu)$ is a composition factor of $\langle F_0,F_1\rangle_G$, hence that of $H^0(\lambda)$.
\end{proof}

We have proved our main theorem.

\begin{tr}\label{blocksoverGL}
Every block $B(\lambda)$ coincides with the equivalence class $\approx(\lambda)$.
\end{tr}

Since the linkage principle for general linear supergroups is now established, the next natural problem we intent to consider in the future is the linkage principle for
orthosymplectic supergroups over the ground field of positive characteristic $p\neq 2$.

\end{document}